\documentclass[]{article}
\usepackage{amssymb, amsmath, amsthm, mathrsfs, bbm, hyperref}
\usepackage[all, color]{xy}
\usepackage[all, table]{xcolor}
\usepackage{enumerate}
\usepackage{tikz}
\usepackage{rotating}

\newtheorem{theorem}{Theorem}

\theoremstyle{definition}
\newtheorem{remark}{Remark}[section]
\newtheorem{proposition}[remark]{Proposition}
\newtheorem{observation}[remark]{Observation}
\newtheorem{definition}[remark]{Definition}

\numberwithin{equation}{section}

\renewcommand{\thetheorem}{\Roman{theorem}}

\usepackage{color}
\def\red{\color{red}}
\def\blue{\color{blue}}
\def\black{\color{black}}

\newcommand{\A}{\mathcal{A}}
\newcommand{\B}{\mathcal{B}}
\newcommand{\M}{\mathcal{M}}
\newcommand{\N}{ \mathrm{N}}
\newcommand{\V}{\mathcal{V}}
\newcommand{\E}{\mathcal{E}}
\newcommand{\W}{\mathcal{W}}
\newcommand{\C}{ \mathbb{C}}
\newcommand{\BY}{ \mathbf{Y}}
\newcommand{\BN}{ \mathbf{N}}

\mathchardef\mhyphen="2D
\newcommand{\verteq}[0]{\begin{turn}{90} $=$ \end{turn}}

\newcommand{\cd}[2][]{\vcenter{\hbox{\xymatrix#1{#2}}}}

\newcommand{\ltwocell  }[3][0.5]{\ar@{}[#2] \ar@{=>}?(#1)+/r 0.2cm/;?(#1)+/l 0.2cm/^{#3}}
\newcommand{\rtwocell  }[3][0.5]{\ar@{}[#2] \ar@{=>}?(#1)+/l 0.2cm/;?(#1)+/r 0.2cm/^{#3}}
\newcommand{\dtwocell  }[3][0.5]{\ar@{}[#2] \ar@{=>}?(#1)+/u 0.2cm/;?(#1)+/d 0.2cm/^{#3}}
\newcommand{\utwocell  }[3][0.5]{\ar@{}[#2] \ar@{=>}?(#1)+/d 0.2cm/;?(#1)+/u 0.2cm/_{#3}}
\newcommand{\congcell  }[2][0.5]{\ar@{}[#2]|(#1){\cong}}

\newcommand{\twosimp}[7]
{
  \cd{
  & {#2} \ar[dr]^{#5} \ar@{}[d]|(0.6){#7} \\
    {#1} \ar[rr]_{#6} \ar[ur]^{#4} & & {#3} }
}

\newcommand{\tricataxiom}[1]{
\begin{gathered}
\def\baselen{#1}
\begin{xy}
	0;<\baselen,0mm>:
	*{\xybox{
		\POS(0,-0.5)*+{\ee}="one"
		\POS(1,1)*+{\aa}="two"
		\POS(3,1)*+{\bb}="three"
		\POS(4,0.0)*+{\cc}="four"
		\POS(0,-1.5)*+{\ff}="five"
		\POS(1,-3)*+{\gg}="six"
		\POS(3,-3)*+{\hh}="seven"
		\POS(4,-2)*+{\ii}="eight"
		\POS(4.5,-1)*+{\dd}="nine"
		\POS(2,0)*+{\jj}="ten"
		\POS(2,-2)*+{\kk}="eleven"
		\POS(3,-1)*+{\ll}="twelve"
		\ar"two";"one"_-{\abcdO}
		\ar"two";"three"^-{\adefO}
		\ar"three";"four"^-{\acdf}
		\ar"one";"five"_-{\abde}
		\ar"five";"six"_-{\bcdeO}
		\ar"six";"seven"_-{\abefO}
		\ar"seven";"eight"_-{\bcef}
		\ar"eight";"nine"_-{\cdefO}
		\ar"four";"nine"^-{\abcfO}
		\ar"one";"ten"|-{\adef}
		\ar"three";"ten"|-{\abcd}
		\ar"ten";"twelve"|-{\abdf}
		\ar"twelve";"nine"|-{\bcdf}
		\ar"five";"eleven"|-{\abef}
		\ar"eleven";"twelve"|-{\bdef}
		\ar"eleven";"seven"|-{\bcde}
		\POS(1.5 ,-1.0)*+{\Downarrow \CCC}="not2"
		\POS(3.1 , 0.0)*+{\Downarrow \EEE}="not4"
		\POS(3.1 ,-2.0)*+{\Downarrow \AAA}="not0"
		\POS(1.5 , 0.5)*+{\cong}="naturaliso1"
		\POS(1.5 ,-2.5)*+{\cong}="naturaliso2"
}}
\end{xy}  \\ \verteq \\
 \def\baselen{#1}
\begin{xy}
0;<\baselen,0mm>:
*{\xybox{
		\POS(1,1)*+{\aa}="one"	
		\POS(3,1)*+{\bb}="two"
		\POS(4,0)*+{\cc}="three"
		\POS(4.5,-1)*+{\dd}="four"
		\POS(0,-0.5)*+{\ee}="five"
		\POS(0,-1.5)*+{\ff}="six"
		\POS(1,-3)*+{\gg}="seven"
		\POS(3,-3)*+{\hh}="eight"
		\POS(4,-2)*+{\ii}="nine"
		\POS(1.5,-1.0)*+{\mm}="ten"
		\POS(3,-1.0)*+{\nn}="eleven"
		\ar"one";"two"^-{\adefO}
		\ar"two";"three"^-{\acdf}
		\ar"three";"four"^-{\abcfO}
		\ar"one";"five"_-{\abcdO}
		\ar"five";"six"_-{\abde}
		\ar"six";"seven"_-{\bcdeO}
		\ar"seven";"eight"_-{\abefO}
		\ar"eight";"nine"_-{\bcef}
		\ar"nine";"four"_-{\cdefO}
		\ar"one";"ten"|-{\acde}
		\ar"ten";"eleven"|-{\acef}
		\ar"eleven";"three"|-{\cdef}
		\ar"ten";"seven"|-{\abce}
		\ar"eleven";"nine"|-{\abcf}
		\POS(2.5 , 0.0)*+{\Downarrow \BBB}="not1"
		\POS(0.75 ,-1.0)*+{\Downarrow \FFF}="not5"
		\POS(2.5 ,-2.0)*+{\Downarrow \DDD}="not3"
		\POS(3.75,-1.0)*+{\cong}="naturaliso3"
}}
\end{xy}
\end{gathered}
}

\newcommand{\stringpent}{
\begin{gathered}
\vcenter{\hbox{\begin{tikzpicture}[y=0.8pt, x=0.8pt,yscale=-1, inner sep=0pt, outer sep=0pt, every text node
part/.style={font=\scriptsize} ]
  \path[draw=black,line join=miter,line cap=butt,line width=0.650pt]
    (0.5761,996.5573) .. controls (25.8299,996.0522) and (35.7165,994.4461) ..
    node[above right=0.12cm,at start] {$\!\!(\cde 1) 1$}(48.0259,1000.9550);
  \path[draw=black,line join=miter,line cap=butt,line width=0.650pt]
    (0.5761,1020.4270) .. controls (32.7994,1020.6795) and (38.4294,1009.3612) ..
    node[above right=0.12cm,at start] {$\!\bce 1$}(48.0259,1003.0477);
  \path[draw=black,line join=miter,line cap=butt,line width=0.650pt]
    (49.8969,1000.5044) .. controls (58.2307,980.6554) and (151.0215,979.4784) ..
    node[above right=0.25cm,pos=0.4] {$\mathfrak{a} 1$}(166.1738,987.8122);
  \path[draw=black,line join=miter,line cap=butt,line width=0.650pt]
    (49.6135,1003.1578) .. controls (60.9777,1006.6933) and (72.9109,1017.9473) ..
    node[below left=0.05cm,pos=0.57]{$\bde 1$} (80.4466,1027.1123);
  \path[draw=black,line join=miter,line cap=butt,line width=0.650pt]
    (0.5761,1042.0698) .. controls (14.4657,1042.3223) and (75.2981,1040.0698) ..
    node[above right=0.12cm,at start] {$\!\abe$} (80.8539,1029.7127);
  \path[draw=black,line join=miter,line cap=butt,line width=0.650pt]
    (82.2825,1030.8845) .. controls (90.1821,1054.0933) and (173.6677,1051.1703)
    .. node[above left=0.12cm,at end] {$\ade$}(214.7867,1051.3122);
  \path[draw=black,line join=miter,line cap=butt,line width=0.650pt]
    (141.5590,1019.2256) .. controls (143.0743,1011.6494) and (166.0258,992.5776)
    .. node[below right=0.04cm] {$1 \mathfrak{a}$}(166.0258,992.5776);
  \path[draw=black,line join=miter,line cap=butt,line width=0.650pt]
    (141.6203,1021.3685) .. controls (141.6203,1021.3685) and (163.4571,1014.0062)
    .. node[above left=0.12cm,at end] {$1 (1 \abc)\!\!$}(214.7868,1014.7638);
  \path[draw=black,line join=miter,line cap=butt,line width=0.650pt]
    (141.4111,1023.0061) .. controls (141.4111,1023.0061) and (142.5493,1033.5038)
    .. node[above left=0.12cm,at end] {$1 \acd$\!}(214.7868,1033.7564);
  \path[draw=black,line join=miter,line cap=butt,line width=0.650pt]
    (167.5590,988.4831) .. controls (167.5590,988.4831) and (177.6142,980.4754) ..
    node[above left=0.12cm,at end] {$\mathfrak{a}$\!} (214.7868,979.9703);
  \path[draw=black,line join=miter,line cap=butt,line width=0.650pt]
    (167.9162,991.3403) .. controls (172.8729,995.3375) and (185.9479,996.5546) ..
    node[above left=0.12cm,at end] {$\mathfrak{a}$\!}(214.7867,996.5546);
  \path[draw=black,line join=miter,line cap=butt,line width=0.650pt]
    (82.5536,1028.3546) .. controls (96.1906,1037.3648) and (132.5520,1035.8496)
    .. node[below right=0.06cm,pos=0.42] {$1 \abd$} (139.8756,1022.0412);
  \path[draw=black,line join=miter,line cap=butt,line width=0.650pt]
    (82.7084,1025.8212) .. controls (84.1609,1008.5230) and (138.1030,990.5459) ..
    node[above left=0.07cm,pos=0.6] {$\mathfrak{a}$}(166.0786,990.2620)
    (49.8145,1001.8228) .. controls (60.2802,1001.0663) and (79.7259,1003.0374) ..
    node[above right=0.08cm,pos=0.4,rotate=-9] {$(1 \bcd) 1$}(98.0839,1006.4270)
    (105.9744,1007.9878) .. controls (121.2679,1011.2244) and (134.6866,1015.3648) ..
    node[below left=0.08cm,pos=0.7,rotate=-15] {$1 (\bcd 1)$}(139.8779,1019.5827);
  \path[fill=black] (166.83624,990.0094) node[circle, draw, line width=0.65pt,
    minimum width=5mm, fill=white, inner sep=0.25mm] (text3313) {$
      \pi$    };
  \path[fill=black] (48.501251,1001.512) node[circle, draw, line width=0.65pt,
    minimum width=5mm, fill=white, inner sep=0.25mm] (text3305) {$\bcde 1$    };
  \path[fill=black] (81.600487,1027.7435) node[circle, draw, line width=0.65pt,
    minimum width=5mm, fill=white, inner sep=0.25mm] (text3309) {$\abde$    };
  \path[fill=black] (145.38065,1021.0309) node[circle, draw, line width=0.65pt,
    minimum width=5mm, fill=white, inner sep=0.25mm] (text3317) {$1 \abcd$  };
\end{tikzpicture}}} \ = \
\vcenter{\hbox{\begin{tikzpicture}[y=0.8pt, x=0.7pt,yscale=-1, inner sep=0pt, outer sep=0pt, every text node part/.style={font=\scriptsize} ]
  \path[draw=black,line join=miter,line cap=butt,line width=0.650pt]
    (0.5761,1010.9221) .. controls (22.7994,1011.1746) and (33.5330,1012.5294) ..
    node[above right=0.12cm,at start] {\!$\bce 1$}(49.1780,1027.2420);
  \path[draw=black,line join=miter,line cap=butt,line width=0.650pt]
    (0.5761,1042.0698) .. controls (14.4657,1042.3223) and (43.6151,1040.3578) ..
    node[above right=0.12cm,at start] {\!$\abe$}(49.1709,1030.0007);
  \path[draw=black,line join=miter,line cap=butt,line width=0.650pt]
    (51.1755,1030.5965) .. controls (67.4280,1056.6856) and (126.7968,1041.6654)
    .. node[below right=0.14cm] {$\ace$} (139.8511,1033.1665);
  \path[draw=black,line join=miter,line cap=butt,line width=0.650pt]
    (142.8500,1030.5854) .. controls (142.8500,1030.5854) and (184.5612,1028.3914)
    .. node[above left=0.12cm,at end] {$1 \acd$\!}(225.8909,1029.1490);
  \path[draw=black,line join=miter,line cap=butt,line width=0.650pt]
    (142.6407,1033.9511) .. controls (142.6407,1033.9511) and (172.9964,1047.0248)
    .. node[above left=0.12cm,at end] {$\ade$\!}(225.8909,1046.9894);
  \path[draw=black,line join=miter,line cap=butt,line width=0.650pt]
    (51.8895,1026.1092) .. controls (64.3077,997.3630) and (142.1189,975.3128) ..
    node[above left=0.12cm,at end] {$\mathfrak{a}$\!}(225.8909,975.6050)
    (52.8868,1028.6427) .. controls (64.4224,1023.3508) and (87.4860,1019.3533) ..
    node[below right=0.1cm,pos=0.53] {$1 \abc$} (113.0660,1016.4776)
    (121.4833,1015.5795) .. controls (135.8648,1014.1237) and (150.7159,1013.0048) ..
    node[above=0.1cm] {$1 \abc$} (164.5150,1012.1936)
    (173.4370,1011.7035) .. controls (191.3809,1010.7870) and (206.7880,1010.4220) ..
    node[above left=0.12cm,at end] {$1 (1 \abc)\!\!$} (225.8909,1010.5363)
    (143.3647,1028.4425) .. controls (173.9707,1013.3776) and (183.0558,992.0340) ..
    node[above left=0.12cm,at end] {$\mathfrak{a}$\!}(225.8909,992.6100)
    (1.1521,981.0038) .. controls (30.9106,980.4472) and (58.1432,987.9204) ..
    node[above right=0.12cm,at start] {$\!\!(\cde 1) 1$}(81.9942,998.1170)
    (88.1242,1000.8286) .. controls (108.0234,1009.9228) and (125.3927,1020.6737) ..
    node[above right=0.1cm,pos=0.22] {$\cde 1$}
    (139.6961,1029.7577);
  \path[fill=black] (49.917496,1028.8956) node[circle, draw, line width=0.65pt,
    minimum width=5mm, fill=white, inner sep=0.25mm] (text3309) {$\abce$    };
  \path[fill=black] (139.03424,1031.976) node[circle, draw, line width=0.65pt,
    minimum width=5mm, fill=white, inner sep=0.25mm] (text3317) {$\acde$  };
\end{tikzpicture}}}
\end{gathered}
}

\begin{document}

\title{The Catalan Simplicial Set II}
\author{Mitchell Buckley}
\date{30 October, 2014}

\maketitle

\begin{abstract}
The Catalan simplicial set $\C$ is known to classify skew-monoidal categories in the sense that a map from $\C$ to a suitably defined nerve of $\mathrm{Cat}$ is precisely a skew-monoidal category~\cite{Catalan1}.
We extend this result to the case of skew monoidales internal to any monoidal bicategory $\B$.
We then show that monoidal bicategories themselves are classified by maps from $\C$ to a suitably defined nerve of $\mathrm{Bicat}$ and extend this result to obtain a definition of skew-monoidal bicategory that aligns with existing theory.
\end{abstract}

\section{Introduction}\label{sec:intro}

Skew-monoidal categories generalise Mac~Lane's notion of monoidal category~\cite{ML1963} by dropping the requirement of invertibility of the associativity and unit constraints.
They were introduced recently by Szlach\'anyi~\cite{Szl2012} in his study of bialgebroids, which are themselves an extension of the notion of quantum group.
Monoidal categories and skew-monoidal categories can be further generalised to notions of \emph{monoidale} and \emph{skew monoidale} in a monoidal bicategory; this has further relevance for quantum algebra, since Lack and Street showed in~\cite{Smswqc} that quantum categories in the sense of~\cite{QCat} can be described using skew monoidales.

Skew-monoidal categories are only one of many possible generalisations of monoidal category: the orientation of the coherence maps and the number and shape of the axioms could reasonably be chosen otherwise.
The connection with bialgebroids and quantum categories motivates the particular generalisation in current usage, but until recently there was no abstract justification for such a choice.

The Catalan simplicial set $\C$ was introduced in~\cite{Catalan1} where it was shown that, apart from a number of interesting combinatorial properties, it classifies skew-monoidal categories in the sense that simplicial maps from $\C$ into a suitably-defined nerve of $\mathrm{Cat}$ are the same thing as skew-monoidal categories.
This provides some abstract justification for the choices made in describing coherence data for skew-monoidal categories.

The first main goal of this paper is to demonstrate that $\C$ has a further classifying property: for any monoidal bicategory $\B$, simplicial maps from $\C$ into a suitably defined nerve of $\B$ are the same as skew monoidales in $\B$.
More precisely, we construct a biequivalence between the $\mathrm{sSet}(\C,\BN\B)$, the bicategory whose objects are simplicial maps from $\C$ to the nerve of $\B$, and $\mathrm{SkMon}(\B)$, the bicategory of skew-monoidales, lax monoidal morphisms and monoidal transformations. 

Our second main goal is to investigate whether $\C$ has this classifying property for higher-dimensional categories; in particular, whether simplicial maps from $\C$ to a suitably defined nerve of $\mathrm{Bicat}$ are the same as monoidal or skew-monoidal bicategories.
We first describe a nerve for $\mathrm{Bicat}$ by informally regarding it as a monoidal tricategory.
We then find that simplicial maps from $\C$ into this nerve contain some unexpected data which go beyond what is required for a monoidal bicategory.
In the case for monoidal bicategories, when the coherence data are invertible, the unexpected data are essentially trivial and the classification result holds.
In the case for skew-monoidal bicategories, the data are not invertible, and the unexpected data appear to be a problem.
We address this by identifying certain simplices in $\C$ and insist that they be mapped to trivial coherence data.
By considering only simplicials maps satisfying this condition, it is easy to compute the data and axioms of a skew-monoidal bicategory.

In Section~\ref{sec:intro} we provide a general introduction.
In Section~\ref{sec:prelim} we define skew-monoidal categories and re-introduce the Catalan simplicial set.
In Section~\ref{sec:monoidal_bicats} we provide an introduction to monoidal bicategories, skew monoidales, and nerves of monoidal bicategories.
In Section~\ref{sec:direct_biequivalence} we describe a biequivalence between maps from $\C$ to the nerve of a monoidal bicategory $\B$ and skew monoidales in $\B$.
In Section~\ref{sec:skew_monoidal_bicategories} we describe a nerve for $\mathrm{Bicat}$, and define skew-monoidal bicategories by examining certain maps from $\C$ to the nerve of $\mathrm{Bicat}$.

\section{Preliminaries}\label{sec:prelim}

In this section we recall the definition of skew-monoidal category, outline our notation for simplicial sets, and re-introduce the Calatan simplical set defined in~\cite{Catalan1}.

\subsection{Skew-monoidal categories}

A \emph{skew-monoidal category} is a category $\A$ equipped with a a unit element $I \in \A$ and a tensor $\otimes \colon \A \times \A \to \A$ with natural families of maps:
\begin{equation}\label{eq:constraints}
\begin{gathered}
  \lambda_A \colon I \otimes A \to A \qquad \text{and} \qquad
  \rho_A \colon A \to A \otimes I \quad \text{(for $A \in \A$)}\\
  \alpha_{ABC} \colon (A \otimes B) \otimes
  C \to A \otimes (B \otimes C)\quad \text{(for $A,B,C, \in
    \A$)}
\end{gathered}
\end{equation}
satisfying five axioms:
\begin{equation}\label{axiom1}
\cd[@!C@C-2cm]{
  & &(A\otimes B)\otimes (C\otimes D) \ar[drr]^{\alpha}& & \\
  ((A\otimes B)\otimes C)\otimes D \ar[urr]^{\alpha} \ar[dr]_{\alpha\otimes D} & & & & A\otimes (B\otimes (C\otimes D) \\
  & (A\otimes (B\otimes C))\otimes D \ar[rr]_{\alpha} && A\otimes ((B\otimes C)\otimes D) \ar[ur]_{A \otimes \alpha}&
}
\end{equation}
\begin{equation}\label{axiom4}
\cd[@!C@C-6pt]{
  & (A\otimes I)\otimes B \ar[r]^{\alpha} & A\otimes (I\otimes B) \ar[dr]^{A \otimes \lambda} & \\
  A\otimes B \ar[ur]^{\rho \otimes B} \ar[rrr]_{\mathrm{id}} & & & A\otimes B
}
\end{equation}
\begin{equation}\label{axiom3}
\cd[@!C@C-12pt]{
  & I\otimes (A\otimes B) \ar[dr]^{\lambda} & \\
  (I\otimes A)\otimes B \ar[ur]^{\alpha} \ar[rr]_{\lambda \otimes B} &  & A\otimes B
}
\end{equation}
\begin{equation}\label{axiom2}
\cd[@!C@C-12pt]{
  & (A\otimes B)\otimes I \ar[dr]^{\alpha} & \\
  A\otimes B \ar[ur]^{\rho} \ar[rr]_{A \otimes \rho} &  & A\otimes (B\otimes I)
}
\end{equation}
\begin{equation}\label{axiom5}
\cd[@!C@C-6pt]{
  & I\otimes I \ar[dr]^{\lambda} & \\
  I \ar[ur]^{\rho} \ar[rr]_{\mathrm{id}} && I\rlap{ .}
}
\end{equation}

These five axioms are the same as those given in Mac\ Lane's original formulation of monoidal categories~\cite{ML1963}. 
Thus, when $\alpha$, $\lambda$ and $\rho$ are invertible, $\A$ is precisely a monoidal category. 
In that case, Kelly~\cite{Kelly1964} showed that the final three axioms can be derived from the first two, in light of which, some definitions of monoidal category choose to include only those first two axioms. 
The same result does not hold for skew-monoidal categories and so we must list all five.

On a similar note, when $\A$ is a monoidal category the commutativity of these particular diagrams in fact implies the commutativity of \emph{all} such diagrams; this is one form of the coherence theorem for monoidal categories~\cite{CWM}.
Skew-monoidal categories, by contrast, do not have the property that all coherence diagrams commute.
For example, the composite $\rho_I\lambda_I \colon I\otimes I \to I\otimes I$ does not generally equal the identity on ${I\otimes I}$.

\subsection{Simplicial sets}\label{subsec:simplicial_sets}

We write $\Delta$ for the simplicial category; the objects are $[n] = \{0,\dots,n\}$ for $n\ge 0$ and the morphisms are order-preserving functions.
Objects $X$ of $\mathrm{SSet} = [\Delta^{\mathrm{op}}, \mathrm{Set}]$ are called {\em simplicial sets}; we write $X_n$ for $X([n])$ and call its elements \emph{$n$-simplices} of $X$.
We use the notation $d_{i}
\colon X_n \to X_{n-1}$ and $s_i \colon X_n \to X_{n+1}$ for the face and degeneracy maps, induced by acting on $X$ by the maps $\delta_i
\colon [n-1]\to [n]$ and $\sigma_{i} \colon [n+1]\to [n]$ of $\Delta$,
the respective injections and surjections for which $\delta_{i}^{-1}(i) = \emptyset$ and $\sigma_i^{-1}(i) = \{i, i+1\}$.
An $(n+1)$-simplex $x$ is called {\em degenerate} when it is in the image of some $s_i$, and \emph{non-degenerate} otherwise.

A simplicial set is called \emph{$r$-coskeletal} when it lies in the image of the right Kan extension functor $[(\Delta^{(r)})^\mathrm{op},
\mathrm{Set}] \to [\Delta^\mathrm{op}, \mathrm{Set}]$, where $\Delta^{(r)} \subset \Delta$ is the full subcategory on those $[n]$ with $n\le r$.
In elementary terms, a simplicial set is $r$-coskeletal when every $n$-boundary with $n > r$ has a unique filler; here, an \emph{$n$-boundary} in a simplicial set is a collection of $(n-1)$-simplices $(x_0, \dots, x_n)$ satisfying $d_j(x_i)
 = d_i(x_{j+1})$ for all $0 \leqslant i \leqslant j < n$;
a \emph{filler} for such a boundary is an $n$-simplex $x$ with $d_i(x)
 = x_i$ for $i = 0, \dots, n$.

\subsection{The Catalan simplicial set}\label{subsec:catalan_definition}

The Catalan simplicial set $\C$ was introduced and studied in~\cite{Catalan1}; its name derives from the fact that it has a Catalan number of simplices in each dimension. 
There are many ways to characterise $\C$ up to isomorphism; perhaps the most concise and elegant is as the nerve of the monoidal poset $(2, \vee, \bot)$. Here, we will take the following description as basic, since it is most helpful for seeing the connection with skew-monoidal categories.

\begin{definition}\label{def:catalan}
The \emph{Catalan simplicial set} $\C$ is the simplicial set with:
\begin{itemize}
  \item A unique $0$-simplex $\star$;
  \item Two $1$-simplices $s_0(\star) \colon \star \to \star$ and
  $c \colon \star \to \star$;
  \item Five $2$-simplices as displayed in:
  \begin{gather*}
    \twosimp{\star}{\star}{\star}{s_0(\star)}{s_0(\star)}{s_0(\star)}{\substack{\,s_0(s_0(\star))\\=s_1(s_0(\star))}} \qquad
    \twosimp{\star}{\star}{\star}{s_0(\star)}{c}{c}{s_0(c)} \qquad
    \twosimp{\star}{\star}{\star}{c}{s_0(\star)}{c}{s_1(c)} \\
    \twosimp{\star}{\star}{\star}{c}{c}{c}{t} \qquad
    \twosimp{\star}{\star}{\star}{s_0(\star)}{s_0(\star)}{c}{i} \rlap{\quad ;}
  \end{gather*}
  \item Higher-dimensional simplices determined by $2$-coskeletality.
  \end{itemize}
\end{definition}
Since $\C$ is 2-coskeletal, all simplices above dimension one are uniquely determined by their faces and as such, every $n$-simplex $a$ for $n \geq 2$ can be identified with the $(n+1)$-tuple of faces $(d_0(a),d_1(a), \dots, d_{n}(a))$.
By direct computation we find that there are four non-degenerate $3$-simplices
\begin{align*} \label{k_is_here}
a &= (t,t,t,t) \\
\ell &= (i,s_1(c),t,s_1(c)) \\
r &= (s_0(c),t,s_0(c),i) \\
k &= (i, s_1(c), s_0(c), i) \ \text{;}
\end{align*}
and nine non-degenerate $4$-simplices
\begin{equation*}
\begin{tabular}{ l l l }
$A1 = (a,a,a,a,a)$ \hspace{0.7cm} &
$A6 = (s_0(i),r, k, \ell, s_2(i))$ \\
$A2 = (r,s_1(t),a,s_1(t),\ell)$ &
$A7 = (k,r, s_0s_1(c), \ell, k)$  \\
$A3 = (r,r,s_2(t),a,s_2(t))$ &
$A8 = (\ell, s_1(t), s_0(t), \ell, k)$ \\
$A4 = (s_0(t),a,s_0(t),\ell,\ell)$ \hspace{0.7cm} &
$A9 = (k, r, s_2(t), s_1(t), r)$ \\
$A5 = (s_1(i),s_2(i), k, s_0(i), s_1(i))$ &
\end{tabular}\ .
\end{equation*}
The simplices above dimension four will play more of a role in Section~\ref{sec:skew_monoidal_bicategories}.

Now consider a simplicial map $F \colon \C \to \N\mathrm{Cat}$, it is completely determined by its behaviour on non-degenerate simplices.
At dimension 0, $F\star$ is the unique 0-simplex in the nerve of $\mathrm{Cat}$.
At dimension 1, we get a category $Fc$.
At dimension 2, we get two functors $Ft \colon Fc \times Fc \to Fc$ and $Fi \colon I \times I \to Fc$.
At dimension 3, we get four natural transformations.
\begin{equation*}
 \cd[@C-1em]{
    (Fc \times Fc) \times Fc \ar[rr]^-{\cong}
    \rtwocell{drr}{Fa} \ar[d]_{Ft \times 1} & &
    Fc \times (Fc \times Fc) \ar[d]^{1 \times Ft} \\
    Fc \times Fc \ar[r]_-{Ft} & Fc & Fc
    \times Fc \ar[l]^-{Ft}
  }
\end{equation*}
\begin{equation*}
 \cd[@C-1em]{
    (I \times I) \times Fc \ar[rr]^-{\cong}
    \rtwocell{drr}{F\ell} \ar[d]_{Fi \times 1} & &
    I \times (I \times Fc) \ar[d]^{1 \times \cong} \\
    Fc \times Fc \ar[r]_-{Ft} & Fc & I
    \times Fc \ar[l]^-{\cong}
  }
\end{equation*}
\begin{equation*}
 \cd[@C-1em]{
    (Fc \times I) \times I \ar[rr]^-{\cong}
    \rtwocell{drr}{Fr} \ar[d]_{\cong \times 1} & &
    Fc \times (I \times I) \ar[d]^{1 \times Fi} \\
    Fc \times I \ar[r]_-{\cong} & Fc & Fc
    \times Fc \ar[l]^-{Ft}
  }
\end{equation*}
\begin{equation*}
 \cd[@C-1em]{
    (I \times I) \times I \ar[rr]^-{\cong}
    \rtwocell{drr}{Fk} \ar[d]_{Fi \times 1} & &
    I \times (I \times I) \ar[d]^{1 \times Fi} \\
    Fc \times I \ar[r]_-{\cong} & I & I
    \times Fc \ar[l]^-{\cong}
  }
\end{equation*}
The unnamed isomorphisms are canonical maps arising from the monoidal category structure on $\mathrm{Cat}$.
Already we can see the strong resemblance with skew-monoidal categories.
At dimension 5 we get nine axioms concerning transformations $Fa$, $F\ell$, $Fr$, $Fk$.
Among those nine are the Mac Lane pentagon and the four other axioms for a skew-monoidal category.

There is some work to do in sorting out the details, but there is a perfect bijection between skew-monoidal categories and simplicial maps $F \colon \C \to \N\mathrm{Cat}$.
This is the final classification result presented in~\cite{Catalan1} and the result which we seek to generalise.

\section{Monoidal bicategories and skew monoidales}\label{sec:monoidal_bicats}

One way to generalise monoidal categories is to consider monoidales in a monoidal bicategory $\B$.
In this case a monoidal category is precisely a monoidale in $\mathrm{Cat}$.
In the same way, it is possible to generalise skew-monoidal categories by describing skew monoidales in a monoidal bicategory $\B$, in which case, a skew-monoidal category is precisely a skew monoidale in $\mathrm{Cat}$.
This generalisation was put to use by Lack and Street in~\cite{Smswqc}, where it was shown that quantum categories in the sense of~\cite{QCat} are skew monoidales in a monoidal bicategory of comodules.

In the following section, we will show that skew-monoiales in a monoidal bicategory $\B$ correspond with simplicial maps from $\C$ into a suitably defined nerve of $\B$. Our result will take the form of a biequivalence
\begin{equation}\label{AAA}
\mathrm{SkMon}(\B) \simeq \mathrm{sSet}(\C,\BN\B)\ .
\end{equation}
The purpose of the present section is to define the bicategories appearing on each side of~\eqref{AAA}. 
We begin by fixing definitions and notation for monoidal bicategories.
We then define skew monoidales and describe the bicategory $\mathrm{SkMon}(\B)$ of skew monoidales in $\B$ appearing to the left of~\eqref{AAA}. 
Finally we describe a nerve construction for monoidal bicategories assigning to each monoidal bicategory $\B$ a simplicial set $\N\B$, and explain how this simplicial set underlies a simplicial bicategory $\BN\B$; now homming into this simplicial bicategory from $\C$ yields the bicategory $\mathrm{sSet}(\C,\BN\B)$ on the right-hand side of~\eqref{AAA}.

\subsection{Monoidal bicategories}

A \emph{monoidal bicategory} is a one-object tricategory in the sense of~\cite{GPS}; it thus comprises a bicategory $\B$ equipped with a unit object $I$ and tensor product homomorphism $\otimes \colon \B \times \B \to \B$ which is associative and unital only up to pseudonatural equivalences $\mathfrak{a}$, $\mathfrak{l}$ and $\mathfrak{r}$.
The coherence of these equivalences is witnessed by invertible modifications $\pi, \mu, \sigma$ and $\tau$, whose components are $2$-cells with boundaries those of the axioms~\eqref{axiom1}--\eqref{axiom2} above, and an invertible $2$-cell $\theta$ whose boundary is that of~\eqref{axiom5}.
The modifications $\pi, \mu, \sigma$ and $\tau$ are as in~\cite{GPS},
though we write $\sigma$ and $\tau$ for what there are called $\lambda$ and $\rho$; whilst $\theta \colon \mathfrak{r}_I \circ \mathfrak{l}_I \Rightarrow 1_{I \otimes I} \colon {I \otimes I} \to {I \otimes I}$ can be defined from the remaining coherence data as the composite
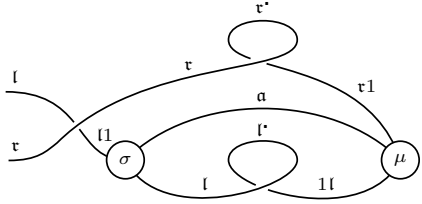
\begin{equation*}
\begin{tikzpicture}[y=0.65pt, x=0.65pt,yscale=-1, inner sep=0pt, outer sep=0pt, every text node part/.style={font=
\scriptsize} ]
  \path[use as bounding box] (-30, 930) rectangle (235,1045);  \path[draw=black,line join=miter,line cap=butt,line
width=0.650pt]
  (40.0000,1012.3622) .. controls (60.0000,1052.3622) and (140.0000,1029.0471) ..
  node[above right=0.1cm,pos=0.37] {$\mathfrak{l}$}(140.0000,1012.3622) .. controls (140.0000,997.2450) and
(100.0000,997.8136) ..
  node[above=0.07cm] {$\mathfrak{l}^\centerdot$}  (100.0000,1012.3622) .. controls (100.0000,1017.3926) and
(106.3215,1022.7556) ..
  (115.8549,1026.9783)(122.9680,1029.7141) .. controls (148.2131,1038.1193) and (186.9261,1038.5100) ..
  node[above left=0.1cm,pos=0.5] {$1 \mathfrak{l}$}(200.0000,1012.3622);
  \path[draw=black,line join=miter,line cap=butt,line width=0.650pt]
  (40.0000,1012.3622) .. controls (80.0000,972.3622) and (160.0000,972.3622) ..
  node[above=0.1cm,pos=0.495] {$\mathfrak{a}$} (200.0000,1012.3622);
  \path[draw=black,line join=miter,line cap=butt,line width=0.650pt]
  (-29.6956,972.6465) .. controls (-7.5878,972.6465) and (2.5465,981.2164) ..
  node[above right=0.12cm,at start] {\!$\mathfrak{l}$}(10.0677,990.3944)
  (13.1931,994.3876) .. controls (20.2119,1003.6506) and (25.9496,1012.3622) ..
  node[above right=0.055cm,pos=0.35] {$\mathfrak{l} 1$}(40.0000,1012.3622)
  (200.0000,1012.3622) .. controls (186.7856,972.7190) and (147.3782,963.6344) ..
  node[above right=0.07cm,pos=0.43] {$\mathfrak{r} 1$}(122.1590,956.2645)
  (112.8208,953.3027) .. controls (105.0036,950.4894) and (100.0000,947.3640) ..
  (100.0000,942.3622) .. controls (100.0000,927.5293) and (140.0000,927.5293) ..
  node[above=0.1cm] {$\mathfrak{r}^\centerdot$}(140.0000,942.3622) .. controls (140.0000,963.1400) and (50.2642,952.4640) ..
  node[above=0.1cm,pos=0.57] {$\mathfrak{r}$} (0.2778,1002.3622) .. controls (-5.7407,1007.0671) and (-10.8148,1012.6400) ..
  node[above right=0.12cm,at end]{\!\!$\mathfrak{r}$}(-28.0000,1012.6400);
  \path[fill=black] (40,1012.3622) node[circle, draw, line width=0.65pt, minimum width=5mm, fill=white, inner sep=0.25mm]
(text2987) {$\sigma$};
  \path[fill=black] (200,1012.3622) node[circle, draw, line width=0.65pt, minimum width=5mm, fill=white, inner sep=0.25mm]
(text2991) {$\mu$};
\end{tikzpicture}
\ .
\end{equation*}
The axioms for a tricategory also imply that each of $\sigma$ and $\tau$ are also completely determined by $\pi$ and $\mu$.

Here, and elsewhere in this paper, we use string notation to display composite $2$-cells in a bicategory, with objects represented by regions, $1$-cells by strings, and generating $2$-cells by vertices.
We orient our string diagrams with 1-cells proceeding down the page and 2-cells proceeding from left to right.
If a 1-cell $\psi$ belongs to a specified adjoint equivalence, then we will denote its specified adjoint pseudoinverse by $\psi^\centerdot$, and as usual with adjunctions, will draw the unit and counit of the adjoint equivalence in string diagrams as simple caps and cups.
In representing the monoidal structure of a bicategory, we notate the tensor product $\otimes$ by juxtaposition, notate the structural 1-cells $\mathfrak a, \mathfrak{l}, \mathfrak{r}$ and 2-cells $\pi, \mu, \sigma, \tau, \theta$ explicitly, and use string crossings to notate pseudonaturality constraint 2-cells, and also instances of the pseudofunctoriality of $\otimes$ of the form $(f \otimes 1) \circ (1 \otimes g) \cong (1
\otimes g) \circ (f \otimes 1)$ (the interchange isomorphisms).
String splittings and joinings are used to notate pseudofunctoriality of $\otimes$ of the form $f \otimes g \cong (f \otimes 1) \circ (1 \otimes g)$ and $(1 \otimes g) \circ (f \otimes 1) \cong f \otimes g $ respectively.

\subsection{Skew monoidales}

Let $\B$ be a monoidal bicategory.

\begin{definition}
A \emph{skew monoidale} in $\B$ is an object $A \in \B$ together with morphisms $i \colon I \to A$ and $t
\colon A \otimes A \to A$, and (non-invertible) coherence $2$-cells
\begin{equation*}
 \cd{
    (A \otimes A) \otimes A \rtwocell{drr}{\alpha} \ar[rr]^{\mathfrak{a}} \ar[d]_{t \otimes A} & &
    A \otimes (A \otimes A) \ar[d]^{A \otimes t} \\
    A \otimes A \ar[r]_t & A & A \otimes A \ar[l]^t
  }
  \ \  \text{and} \ \
  \cd{
    I \otimes A \ar[d]_{i \otimes A} \rtwocell{dr}{\lambda}
    \ar[r]^-{\mathfrak{l}} &
    A \rtwocell{dr}{\rho} \ar@{=}[d] \ar[r]^-{\mathfrak{r}}
    & A \otimes I \ar[d]^{A \otimes i} \\
    A \otimes A \ar[r]_-t & A & A \otimes A \ar[l]^-t
  }
\end{equation*}
subject to the following five axioms, the appropriate analogues of~\eqref{axiom1}--\eqref{axiom5}.
\begin{align*}
\vcenter{\hbox{
}}
\end{align*}
\end{definition}

\begin{definition}
Let $A$ and $B$ be skew monoidales in $\B$.
A \emph{lax monoidal morphism from $A$ to $B$} consists of a morphism
$F \colon A \to B$ together with 2-cells
\begin{equation*}
\cd[]{
A \otimes A \ar[r]^-{t} \ar[d]_{F \otimes F} & A \ar[d]^{F} \\
B \otimes B \ar[r]_-{t} \rtwocell{ur}{\phi} & B
}
\quad \text{and} \quad
\cd[]{
 I \ar[r]^{i} \ar@{=}[d] & A \ar[d]^{F} \\
 I \ar[r]_{i} \rtwocell{ur}{\psi} & B
}
\end{equation*}
subject to the following three axioms.
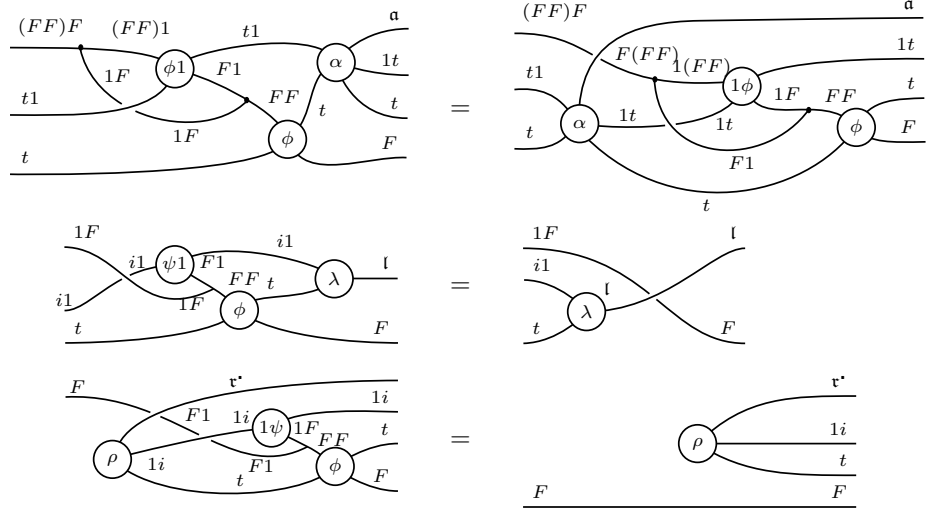
\begin{align*}
\vcenter{\hbox{%
}}
\end{align*}
\end{definition}

\begin{definition}
Let $F$ and $G$ be lax monoidal morphisms from $A$ to $B$.
A \emph{monoidal transformation from $F$ to $G$} consists of a 2-cell $\gamma \colon F \Rightarrow G$ satisfying the following two axioms.
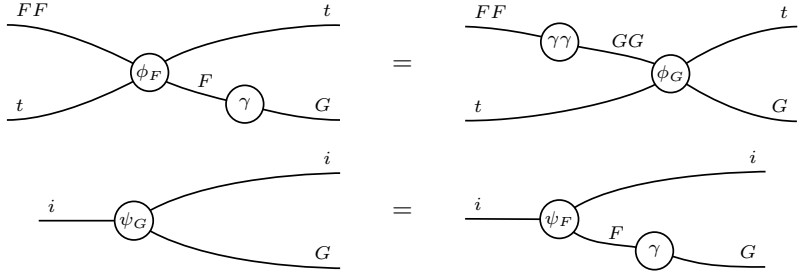
\begin{align*}
\vcenter{\hbox{
%
}}
\end{align*}
\end{definition}

Together, skew monoidales, lax monoidal morphisms, and monoidal transformations in $\B$ form a bicategory $
\mathrm{SkMon}(\B)$.
Suppose $(F, \phi_F, \psi_F)\colon~A~\to~B$ and $(G, \phi_G, \psi_G) \colon B \to C$ are lax monoidal morphisms; their composite is $GF$ together with 2-cells
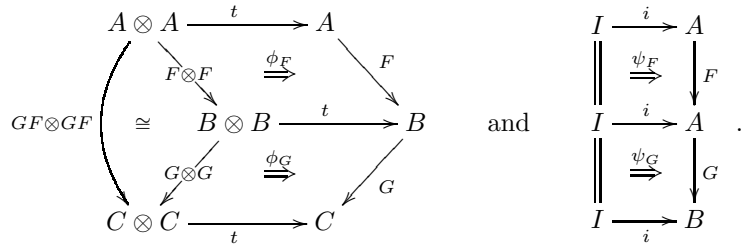
\begin{equation*}
\cd[@!C@C-24pt]{
A \otimes A \ar[rr]^{t} \ar[dr]|{F \otimes F} \ar@/_16pt/[dd]_{GF \otimes GF} && A \ar[dr]^{F} & \\
& B \otimes B \ar[rr]^{t} \ar[dl]|{G \otimes G} \rtwocell{ur}{\phi_F} && B \ar[dl]^{G} \\
C \otimes C \ar[rr]_{t} \rtwocell{urrr}{\phi_G} \congcell{uu} && C &
}
\qquad \text{and} \qquad
\cd[]{
I \ar[r]^{i} \ar@{=}[d] & A \ar[d]^{F} \\
I \ar[r]^{i} \ar@{=}[d] \rtwocell{ur}{\psi_F} & A \ar[d]^{G} \\
I \ar[r]_{i} \rtwocell{ur}{\psi_G} & B
}\ .
\end{equation*}
The unnamed isomorphism arises from the pseudo-functoriality of $\otimes : \B \times \B \to \B$.
The identity morphism on a skew monoidale $A$ is $1_A \colon A \to A$ together with 2-cells
\begin{equation*}
\cd[@C+8pt@R+8pt]{
A \otimes A \ar[r]^{t} \ar@/^10pt/[d]^{1} \ar@/_10pt/[d]_{1 \otimes 1} \congcell{d} \ar@/^4pt/[dr]|{t} & A \ar[d]^{1}
\\
A \otimes A \ar[r]_{t} \congcell[0.4]{ur} \congcell[0.7]{ur} & A
}
\qquad \text{and} \qquad
\cd[@C+8pt@R+8pt]{
I \ar[r]^{i} \ar@{=}[d] \ar[dr]|{i} & A \ar[d]^{1} \\
I \ar[r]_{i}  \congcell[0.7]{ur}  & A
}\ .
\end{equation*}
The unnamed isomorphisms arise from the pseudo-functoriality of $\otimes : \B \times \B \to \B$ and coherence cells in the bicategory.
If $\alpha$ and $\beta$ are composable transformations, their composite is $\beta \alpha$.
The identity transformation on a lax monoidal morphism $F$ is $1_F$.
Coherence 2-cells for $\mathrm{SkMon}(\B)$ are inherited from $\B$.
That is, if $F,G,H$ are composable lax monoidal morphisms then the coherence isomorphisms $(HG)F \cong H(GF)$, $1F \cong F$ and $F \cong F1$ in $\B$ are already monoidal transformations.

\subsection{Nerves of monoidal bicategories}\label{subsec:nerves}

As noted above, a monoidal bicategory is a one-object tricategory in the sense of~\cite{GPS}.
There are several known constructions of nerves for tricategories; the one of interest to us is essentially Street's $\omega$-categorical nerve~\cite{aos}, restricted from dimension $\omega$ to dimension $3$, and generalised from strict to weak $3$-categories.
An explicit description of this nerve is given in~\cite{GRoT}; we now reproduce the details for the case of a monoidal bicategory $\B$.

\begin{definition}
The \emph{nerve of $\B$}, $\N\B$, is the simplicial set with:
\begin{itemize}
\item A unique $0$-simplex $\star$.
\item A $1$-simplex is an object $A_{01}$ of $\B$; its two faces are necessarily $\star$.
\item A $2$-simplex is given by objects $A_{12}, A_{02}, A_{01}$ of $\B$ together with a $1$-cell $${A}_{012} \colon {A}_{12} \otimes {A}_{01} \to {A}_{02}\ ;$$ its three faces are $A_{12}$, $A_{02}$, and $A_{01}$.
\item A $3$-simplex is given by:
\begin{itemize}
\item Objects ${A}_{ij}$ for each $0 \leqslant i < j
  \leqslant 3$;
\item $1$-cells ${A}_{ijk} \colon {A}_{jk} \otimes {A}_{ij} \to {A}_{ik}$ for
  each $0 \leqslant i < j < k \leqslant 3$;
\item A $2$-cell
  \begin{equation*}
 \cd{
    ({A}_{23} \otimes {A}_{12}) \otimes {A}_{01} \rtwocell{drr}{{A}_{0123}} \ar[rr]^{\mathfrak{a}} \ar[d]_{{A}_{123} \otimes 1}
& &
    {A}_{23} \otimes ({A}_{12} \otimes {A}_{01}) \ar[d]^{1 \otimes {A}_{012}} \\
    {A}_{13} \otimes {A}_{01} \ar[r]_{{A}_{013}} & {A}_{03} & {A}_{23}
    \otimes {A}_{02}\rlap{ ;} \ar[l]^{{A}_{023}}
  }
  \end{equation*}
\end{itemize}
its four faces are $A_{123}$, $A_{023}$, $A_{013}$ and $A_{012}$.
\item A $4$-simplex is given by:
\begin{itemize}
\item Objects ${A}_{ij}$ for each $0 \leqslant i < j
  \leqslant 4$;
\item  $1$-cells ${A}_{ijk} \colon {A}_{jk} \otimes {A}_{ij} \to {A}_{ik}$ for
  each $0 \leqslant i < j < k \leqslant 4$;
\item $2$-cells ${A}_{ijk\ell} \colon {A}_{ij\ell} \circ ({A}_{jk\ell}
  \otimes 1) \Rightarrow {A}_{ik\ell} \circ (1 \otimes {A}_{ijk}) \circ
  \mathfrak{a}$ for each $0 \leqslant i < j < k < \ell \leqslant 4$
\end{itemize}
such that the $2$-cell equality
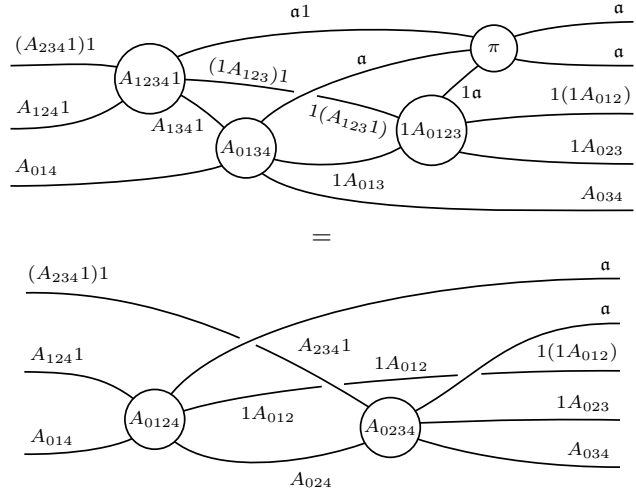
\begin{gather*}\label{NB4simplex}
\vcenter{\hbox{\begin{tikzpicture}[y=1pt, x=1.1pt,yscale=-1, inner sep=0pt, outer sep=0pt, every text node part/.style={font=
\scriptsize} ]
  \path[draw=black,line join=miter,line cap=butt,line width=0.650pt]
    (0.5761,996.5573) .. controls (25.8299,996.0522) and (35.7165,994.4461) ..
    node[above right=0.15cm,at start] {$\!\!({A}_{234} 1) 1$}(48.0259,1000.9550);
  \path[draw=black,line join=miter,line cap=butt,line width=0.650pt]
    (0.5761,1020.4270) .. controls (32.7994,1020.6795) and (38.4294,1009.3612) ..
    node[above right=0.17cm,at start] {$\!\!{A}_{124} 1$}(48.0259,1003.0477);
  \path[draw=black,line join=miter,line cap=butt,line width=0.650pt]
    (49.8969,1000.5044) .. controls (58.2307,980.6554) and (151.0215,979.4784) ..
    node[above right=0.25cm,pos=0.4] {$\mathfrak{a} 1$}(166.1738,987.8122);
  \path[draw=black,line join=miter,line cap=butt,line width=0.650pt]
    (49.6135,1003.1578) .. controls (60.9777,1006.6933) and (72.9109,1017.9473) ..
    node[below left=0.05cm,pos=0.57]{${A}_{134} 1$} (80.4466,1027.1123);
  \path[draw=black,line join=miter,line cap=butt,line width=0.650pt]
    (0.5761,1042.0698) .. controls (14.4657,1042.3223) and (75.2981,1040.0698) ..
    node[above right=0.15cm,at start] {$\!\!{A}_{014}$} (80.8539,1029.7127);
  \path[draw=black,line join=miter,line cap=butt,line width=0.650pt]
    (82.2825,1030.8845) .. controls (90.1821,1054.0933) and (173.6677,1051.1703)
    .. node[above left=0.12cm,at end] {${A}_{034}$}(214.7867,1051.3122);
  \path[draw=black,line join=miter,line cap=butt,line width=0.650pt]
    (141.5590,1019.2256) .. controls (143.0743,1011.6494) and (166.0258,992.5776)
    .. node[below right=0.04cm] {$1 \mathfrak{a}$}(166.0258,992.5776);
  \path[draw=black,line join=miter,line cap=butt,line width=0.650pt]
    (141.6203,1021.3685) .. controls (141.6203,1021.3685) and (163.4571,1014.0062)
    .. node[above left=0.12cm,at end] {$1 (1 {A}_{012})\!\!$}(214.7868,1014.7638);
  \path[draw=black,line join=miter,line cap=butt,line width=0.650pt]
    (141.4111,1023.0061) .. controls (141.4111,1023.0061) and (142.5493,1033.5038)
    .. node[above left=0.12cm,at end] {$1 {A}_{023}$}(214.7868,1033.7564);
  \path[draw=black,line join=miter,line cap=butt,line width=0.650pt]
    (167.5590,988.4831) .. controls (167.5590,988.4831) and (177.6142,980.4754) ..
    node[above left=0.12cm,at end] {$\mathfrak{a}$} (214.7868,979.9703);
  \path[draw=black,line join=miter,line cap=butt,line width=0.650pt]
    (167.9162,991.3403) .. controls (172.8729,995.3375) and (185.9479,996.5546) ..
    node[above left=0.12cm,at end] {$\mathfrak{a}$}(214.7867,996.5546);
  \path[draw=black,line join=miter,line cap=butt,line width=0.650pt]
    (82.5536,1028.3546) .. controls (96.1906,1037.3648) and (132.5520,1035.8496)
    .. node[below right=0.12cm,pos=0.42] {$1 {A}_{013}$} (139.8756,1022.0412);
  \path[draw=black,line join=miter,line cap=butt,line width=0.650pt]
    (82.7084,1025.8212) .. controls (84.1609,1008.5230) and (138.1030,990.5459) ..
    node[above left=0.07cm,pos=0.6] {$\mathfrak{a}$}(166.0786,990.2620)
    (49.8145,1001.8228) .. controls (60.2802,1001.0663) and (79.7259,1003.0374) ..
    node[above right=0.08cm,pos=0.4,rotate=-9] {$(1 {A}_{123}) 1$}(98.0839,1006.4270)
    (105.9744,1007.9878) .. controls (121.2679,1011.2244) and (134.6866,1015.3648) ..
    node[below left=0.08cm,pos=0.7,rotate=-15] {$1 ({A}_{123} 1)$}(139.8779,1019.5827);
  \path[fill=black] (166.83624,990.0094) node[circle, draw, line width=0.65pt,
    minimum width=5mm, fill=white, inner sep=0.25mm] (text3313) {$\ \
      \pi\ \ $    };
  \path[fill=black] (48.501251,1001.512) node[circle, draw, line width=0.65pt,
    minimum width=5mm, fill=white, inner sep=0.25mm] (text3305) {${A}_{1234} 1$    };
  \path[fill=black] (81.600487,1027.7435) node[circle, draw, line width=0.65pt,
    minimum width=5mm, fill=white, inner sep=0.25mm] (text3309) {${A}_{0134}$    };
  \path[fill=black] (145.38065,1021.0309) node[circle, draw, line width=0.65pt,
    minimum width=5mm, fill=white, inner sep=0.25mm] (text3317) {$1 {A}_{0123}$  };
\end{tikzpicture}}} \\ = \\
\vcenter{\hbox{\begin{tikzpicture}[y=1pt, x=1pt,yscale=-1, inner sep=0pt, outer sep=0pt, every text node part/.style={font=
\scriptsize} ]
  \path[draw=black,line join=miter,line cap=butt,line width=0.650pt]
    (0.5761,1010.9221) .. controls (22.7994,1011.1746) and (33.5330,1012.5294) ..
    node[above right=0.12cm,at start] {\!${A}_{124} 1$}(49.1780,1027.2420);
  \path[draw=black,line join=miter,line cap=butt,line width=0.650pt]
    (0.5761,1042.0698) .. controls (14.4657,1042.3223) and (43.6151,1040.3578) ..
    node[above right=0.12cm,at start] {\!${A}_{014}$}(49.1709,1030.0007);
  \path[draw=black,line join=miter,line cap=butt,line width=0.650pt]
    (51.1755,1030.5965) .. controls (67.4280,1056.6856) and (126.7968,1041.6654)
    .. node[below right=0.14cm] {${A}_{024}$} (139.8511,1033.1665);
  \path[draw=black,line join=miter,line cap=butt,line width=0.650pt]
    (142.8500,1030.5854) .. controls (142.8500,1030.5854) and (184.5612,1028.3914)
    .. node[above left=0.12cm,at end] {$1 {A}_{023}$}(225.8909,1029.1490);
  \path[draw=black,line join=miter,line cap=butt,line width=0.650pt]
    (142.6407,1033.9511) .. controls (142.6407,1033.9511) and (172.9964,1047.0248)
    .. node[above left=0.12cm,at end] {${A}_{034}$}(225.8909,1046.9894);
  \path[draw=black,line join=miter,line cap=butt,line width=0.650pt]
    (51.8895,1026.1092) .. controls (64.3077,997.3630) and (142.1189,975.3128) ..
    node[above left=0.12cm,at end] {$\mathfrak{a}$}(225.8909,975.6050)
    (52.8868,1028.6427) .. controls (64.4224,1023.3508) and (87.4860,1019.3533) ..
    node[below right=0.1cm,pos=0.53] {$1 {A}_{012}$} (113.0660,1016.4776)
    (121.4833,1015.5795) .. controls (135.8648,1014.1237) and (150.7159,1013.0048) ..
    node[above=0.1cm] {$1 {A}_{012}$} (164.5150,1012.1936)
    (173.4370,1011.7035) .. controls (191.3809,1010.7870) and (206.7880,1010.4220) ..
    node[above left=0.12cm,at end] {$1 (1 {A}_{012})\!\!$} (225.8909,1010.5363)
    (143.3647,1028.4425) .. controls (173.9707,1013.3776) and (183.0558,992.0340) ..
    node[above left=0.12cm,at end] {$\mathfrak{a}$}(225.8909,992.6100)
    (1.1521,981.0038) .. controls (30.9106,980.4472) and (58.1432,987.9204) ..
    node[above right=0.12cm,at start] {$\!\!({A}_{234} 1) 1$}(81.9942,998.1170)
    (88.1242,1000.8286) .. controls (108.0234,1009.9228) and (125.3927,1020.6737) ..
    node[above right=0.1cm,pos=0.22] {${A}_{234} 1$}
    (139.6961,1029.7577);
  \path[fill=black] (49.917496,1028.8956) node[circle, draw, line width=0.65pt,
    minimum width=5mm, fill=white, inner sep=0.25mm] (text3309) {${A}_{0124}$    };
  \path[fill=black] (139.03424,1031.976) node[circle, draw, line width=0.65pt,
    minimum width=5mm, fill=white, inner sep=0.25mm] (text3317) {${A}_{0234}$  };
\end{tikzpicture}}}
\end{gather*}
holds.
The five faces of this simplex are $A_{1234}$, $A_{0234}$,
$A_{0134}$, $A_{0124}$ and $A_{0123}$.
\item Higher-dimensional simplices are determined by the requirement
  that $\N\B$ be $4$-coskeletal.
\end{itemize}
It remains to describe the degeneracy operators.
The degeneracy of the unique $0$-simplex is the unit object $I \in \B$; the two degeneracies $s_0(A), s_1(A)$ of a $1$-simplex $A \in \B$ are the unit constraints $\mathfrak r^\centerdot \colon A \otimes I \to A$ and $\mathfrak{l} \colon I
\otimes A \to A$; the three degeneracies $s_0(\gamma), s_1(\gamma)$ and $s_2(\gamma)$ of a $2$-simplex $\gamma \colon B \otimes C \to A$ are the respective $2$-cells
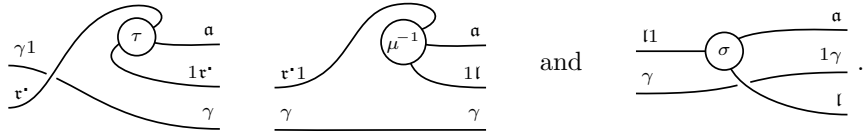
\begin{equation*}
\vcenter{\hbox{\begin{tikzpicture}[y=0.8pt, x=0.8pt,yscale=-1, inner sep=0pt, outer sep=0pt, every text node
part/.style={font=\scriptsize} ]
  \path[draw=black,line join=miter,line cap=butt,line width=0.650pt]
  (0.0000,902.3622) .. controls (22.2057,902.3622) and (24.4296,862.1105) ..
  node[above right=0.12cm,at start] {$\!\mathfrak{r}^\centerdot$}(54.6378,854.8261) .. controls (70.4902,851.0034) and
(80.9687,860.6367) .. (63.1815,866.8469)
  (0.0000,882.3622) .. controls (6.5599,882.4187) and (12.3710,883.9537) ..
  node[above right=0.12cm,at start] {$\!\gamma 1$}(18.0721,886.3023)(22.7007,888.3801) .. controls (40.8765,897.1386) and
(59.6979,912.3622) ..
  node[above left=0.12cm,at end] {$\gamma$\!}(100.0000,912.3622);
  \path[draw=black,line join=miter,line cap=butt,line width=0.650pt] (59.1409,868.1591) .. controls (29.0888,882.3013) and
(70.2005,892.3622) ..
  node[above left=0.12cm,at end] {$1 \mathfrak{r}^\centerdot$\!}(100.0000,892.3622);
  \path[draw=black,line join=miter,line cap=butt,line width=0.650pt] (62.6764,869.3723) .. controls (70.2525,874.1706) and
(91.3989,872.3622) ..
  node[above left=0.12cm,at end] {$\mathfrak{a}$\!}(100.0000,872.3622);
  \path[fill=black] (60.710678,868.80756) node[circle, draw, line width=0.65pt, minimum width=5mm, fill=white, inner
sep=0.25mm] (text14966) {$\tau$   };
 \end{tikzpicture}}}
\qquad
\vcenter{\hbox{\begin{tikzpicture}[y=0.8pt, x=0.8pt,yscale=-1, inner sep=0pt, outer sep=0pt, every text node
part/.style={font=\scriptsize} ]
  \path[draw=black,line join=miter,line cap=butt,line width=0.650pt] (0.0000,892.3622) .. controls (36.8530,892.1096) and
(33.3645,857.8413) ..
  node[above right=0.12cm,at start] {\!$\mathfrak{r}^\centerdot 1$}(56.4056,853.3108) .. controls (75.1348,849.6282) and
(85.9633,861.3740) .. (63.1815,866.8469);
  \path[draw=black,line join=miter,line cap=butt,line width=0.650pt] (0.0000,912.3622) .. controls (29.3000,912.6147) and
(43.6628,912.3622) ..
  node[above left=0.12cm,at end] {$\gamma$\!}node[above right=0.12cm,at start] {\!$\gamma$}(100.0000,912.3622);
  \path[draw=black,line join=miter,line cap=butt,line width=0.650pt] (62.5254,872.1096) .. controls (63.3845,886.3127) and
(70.2005,892.3622) ..
  node[above left=0.12cm,at end] {$1 \mathfrak{l}$\!}(100.0000,892.3622);
  \path[draw=black,line join=miter,line cap=butt,line width=0.650pt] (62.6764,869.3723) .. controls (70.2525,874.1706) and
(91.3989,872.3622) ..
  node[above left=0.12cm,at end] {$\mathfrak{a}$\!}(100.0000,872.3622);
  \path[fill=black] (60.963215,870.82788) node[circle, draw, line width=0.65pt, minimum width=5mm, fill=white, inner
sep=0.25mm] (text14966) {$\mu^{\!\scriptscriptstyle -1}$
   };
 \end{tikzpicture}}}\qquad\text{and}\qquad
\vcenter{\hbox{\begin{tikzpicture}[y=0.8pt, x=0.8pt,yscale=-1, inner sep=0pt, outer sep=0pt, every text node
part/.style={font=\scriptsize} ]
  \path[draw=black,line join=miter,line cap=butt,line width=0.650pt]
  (43.0799,884.7365) .. controls (43.9391,898.9396) and (70.2005,912.3622) ..
  node[above left=0.12cm,at end] {$\mathfrak{l}$\!}(100.0000,912.3622)
  (0.0000,902.3622) .. controls (26.8739,902.5938) and (38.5571,900.6887) ..
  node[above right=0.12cm,at start] {\!$\gamma$}(48.0003,898.4810)
  (54.0398,897.0009) .. controls (63.6207,894.6257) and (73.8763,892.3622) ..
  node[above left=0.12cm,at end] {$1 \gamma$\!}(100.0000,892.3622);
  \path[draw=black,line join=miter,line cap=butt,line width=0.650pt] (43.2310,880.4840) .. controls (49.3612,869.8661) and
(91.3989,872.3622) ..
  node[above left=0.12cm,at end] {$\mathfrak{a}$\!}(100.0000,872.3622);
  \path[draw=black,line join=miter,line cap=butt,line width=0.650pt] (0.0000,882.3622) --
  node[above right=0.12cm,at start] {\!$\mathfrak{l} 1$}(40.0000,882.3622);
  \path[fill=black] (41.767765,882.36218) node[circle, draw, line width=0.65pt, minimum width=5mm, fill=white, inner
sep=0.25mm] (text14966) {$\sigma$
     };
 \end{tikzpicture}}}\rlap{ .}
\end{equation*}
The four degeneracies of a $3$-simplex are simply the assertions of certain $2$-cell equalities; that these hold
is a consequence of the axioms for a monoidal bicategory.
Higher degeneracies are determined by coskeletality.

All simplicial identities except $s_0(I) = s_1(I)$ (i.e.\ $r^\centerdot_I = \ell_I$) hold automatically.
There is however a canonical isomorphism $r^\centerdot_I \cong \ell_I$, see~\cite{Gurski} A.3.1.
Thus $\ell_I$ is a pseudo-inverse for $r_I$ and we can suppose that $r^\centerdot_I = \ell_I$ without any loss of generality.
\end{definition}

\begin{definition}
The \emph{pseudo nerve of $\B$}, called $\N_{\mathrm{p}}\B$, is the same as $\N\B$ with the extra requirement that 3-simplex components $A_{0123}$ be invertible.
\end{definition}

\begin{remark} \label{nerve_right_adjoint}
The assignation $\B \mapsto \N(\B)$ sending a monoidal bicategory to its nerve can be extended to a functor $\N \colon \mathrm{MonBicat}_s \to \mathrm{SSet}$, where $\mathrm{MonBicat}_s$ is the category of monoidal bicategories and morphisms which strictly preserve all the structure.
When seen in this way, the nerve is a right adjoint.
This holds equally well for $\N_{\mathrm{p}}$.
\end{remark}

Now that the nerve is well defined, we can properly examine simplicial maps $F \colon \C \to \N\B$. 
In the lowest few dimensions the data for such an $F$ consist of the following.
\begin{itemize}
\item A single object $Fc$ in $\B$.
\item Two 1-cells in $\B$
\begin{equation*}
\cd[]{Fc\otimes Fc \ar[r]^-{Ft} & Fc}\quad\text{and}\quad\cd[]{I\otimes I \ar[r]^-{Fi} & Fc}
\end{equation*}
since $F(s_0(\star)) = I$.
\end{itemize}
Before we can examine the higher data we already notice a problem: while $Ft$ has the right form to provide a multiplication $Fc \otimes Fc \to Fc$, the map $Fi \colon I \otimes I \to Fc$ has the wrong domain to be a unit map for $Fc$.
While this problem is easily resolved using the canonical equivalence of $I \otimes I$ with $I$, the fact that $I \otimes I$ and $I$ are only \emph{equivalent} and not \emph{isomorphic} means that the correspondence we're investigating cannot be a literal bijection between $\mathrm{sSet}(\C,\N\B)$ and the set of skew monoidales in $\B$.
It will, however, be surjective up to equivalence when $\mathrm{sSet}(\C,\N\B)$ is regarded as a \emph{bicategory}.

\subsection{The bicategory $\mathrm{sSet}(\C,\BN\B)$}

In order to construct the bicategory $\mathrm{sSet}(\C, \BN\B)$, we will first show that $\N\B$ underlies a simplicial bicategory (a bicategory object internal to simplicial sets). 
Then since the representable $\mathrm{sSet}(\C, \mhyphen) \colon \mathrm{sSet} \to \mathrm{Set}$ preserves limits, $\mathrm{sSet}(\C,\N\B)$ becomes the set of objects of a bicategory.

\begin{observation} \label{obs:simpl_bicat}
The nerve of a monoidal bicategory $\N\B$ is the object of objects of a bicategory internal to $\mathrm{sSet}$
\begin{equation} \label{nerve_internal_bicat}
 \cd[]{
  \N(\B \Downarrow \B) \ar@<-.5ex>[r] \ar@<.5ex>[r] &
  \N(\B \downarrow \B) \ar@<-.5ex>[r] \ar@<.5ex>[r] &
  \N(\B) }
 \end{equation}
where $(\B \downarrow \B)$ and $(\B \Downarrow \B)$ are monoidal bicategories defined below.
We call this internal bicategory $\BN\B$; see Table~\ref{table1} for an explicit description of 0, 1 and 2-cells in the lowest few dimensions.
We construct it by first building a bicategory
\begin{equation*}
 \cd[]{
  (\B \Downarrow \B) \ar@<-.5ex>[r] \ar@<.5ex>[r] &
  (\B \downarrow \B) \ar@<-.5ex>[r] \ar@<.5ex>[r] &
  (\B) }
 \end{equation*}
internal to $\mathrm{MonBicat}_s$ and then use the fact that $\N$ preserves limits (because it is a right adjoint, Remark~\ref{nerve_right_adjoint}).

The \emph{oplax-comma monoidal bicategory} $(\B \downarrow \B)$ is defined as follows.
Its objects are arrows $h \colon A \to B$.
A morphism from $h$ to $h' \colon A' \to B'$ is a triple $(f_A, f_B, f_{h})$, where
$f_A \colon A \to A'$,
$f_B \colon B \to B'$, and where
\begin{equation}\label{comma_one_cell}
\cd[]{
A \ar[d]_{h} \ar[r]^{f_A} & A' \ar[d]^{h'} \\
B \ar[r]_{f_B} \rtwocell{ur}{f_{h}} & B'
}
\end{equation}
A 2-cell from $(f_A, f_B, f_{h})$ to $(g_A, g_B, g_{h})$ is a pair $(\alpha_{A},\alpha_{B})$, where $\alpha_{A}\colon f_A\Rightarrow g_A$ and $\alpha_B\colon f_B \Rightarrow g_B$ satisfy
\begin{equation}\label{comma_two_cell}
\vcenter{\hbox{
\begin{tikzpicture}[y=0.60pt, x=0.6pt, yscale=-1, xscale=-1, inner sep=0pt, outer sep=0pt, every text node part/.style={font=
\scriptsize} ]
\path[draw=black,line join=miter,line cap=butt,line width=0.650pt]
  (210.0000,902.3622) .. controls (250.0000,902.3622) and (300.0000,872.3622) ..
  node[above left=0.12cm, at start] {$h'$} (300.0000,872.3622);
\path[draw=black,line join=miter,line cap=butt,line width=0.650pt]
  (210.0000,842.3622) .. controls (250.0000,842.3622) and (300.0000,872.3622) ..
  node[above left=0.12cm, at start] {$g_A$} (300.0000,872.3622);
\path[draw=black,line join=miter,line cap=butt,line width=0.650pt]
  (420.0000,902.3622) .. controls (390.0000,902.3622) and (360.0000,892.3622) ..
  node[above right=0.12cm, at start] {$f_B$} (360.0000,892.3622);
\path[draw=black,line join=miter,line cap=butt,line width=0.650pt]
  (301.7391,871.0578) .. controls (321.9565,852.1448) and (390.0000,841.4926) ..
  node[above right=0.12cm, at end] {$h$} (420.0000,842.5796);
\path[draw=black,line join=miter,line cap=butt,line width=0.650pt]
  (301.5217,873.2318) .. controls (314.1304,883.0144) and (358.2609,892.1448) ..
  node[above=0.12cm, pos=0.5] {$g_B$} (358.2609,892.1448);
\path[fill=black] (360,892.36218) node[circle, draw, line width=0.65pt, minimum
  width=5mm, fill=white, inner sep=0.25mm] (text3097) {$\alpha_B$    };
\path[fill=black] (300,872.36218) node[circle, draw, line width=0.65pt, minimum
  width=5mm, fill=white, inner sep=0.25mm] (text3101) {$g_h$    };
\end{tikzpicture}
}}
\quad \ = \ \quad
\vcenter{\hbox{
\begin{tikzpicture}[y=0.60pt, x=0.6pt,yscale=-1, xscale=-1, inner sep=0pt, outer sep=0pt, every text node part/.style={font=
\scriptsize} ]
\path[draw=black,line join=miter,line cap=butt,line width=0.650pt]
  (341.7391,973.6665) .. controls (353.4783,981.4926) and (385.6522,1002.1448)
  .. node[above right=0.12cm, at end] {$f_B$} (419.5652,1002.3622);
\path[draw=black,line join=miter,line cap=butt,line width=0.650pt]
  (338.9130,973.4491) .. controls (323.6957,987.5796) and (257.1739,1002.3622)
  .. node[above left=0.12cm, at end] {$h'$} (210.4348,1001.9274);
\path[draw=black,line join=miter,line cap=butt,line width=0.650pt]
  (342.1739,971.4926) .. controls (349.7826,964.3187) and (385.0000,943.0144) ..
  node[above right=0.12cm, at end] {$h$} (419.7826,942.5796);
\path[draw=black,line join=miter,line cap=butt,line width=0.650pt]
  (338.6957,970.4057) .. controls (325.0000,961.0578) and (290.8696,957.7970) ..
  node[above=0.12cm, pos=0.5] {$f_A$} (272.3913,953.2318);
\path[draw=black,line join=miter,line cap=butt,line width=0.650pt]
  (267.8261,951.2752) .. controls (267.8261,951.2752) and (240.2174,943.4491) ..
  node[above left=0.12cm, at end] {$g_A$} (210.4348,942.5796);
\path[fill=black] (340,972.36218) node[circle, draw, line width=0.65pt, minimum
  width=5mm, fill=white, inner sep=0.25mm] (text3240-1-9-2) {$f_h$    };
\path[fill=black] (270,952.36218) node[circle, draw, line width=0.65pt, minimum
  width=5mm, fill=white, inner sep=0.25mm] (text3057) {$\alpha_A$    };
\end{tikzpicture}
}}\ .
\end{equation}
Composition and identities are defined in the obvious way.
The tensor for the monoidal structure is defined on 0 and 2-cells by tensoring the underlying data in $\B$.
On 1-cells, we need $(f_A, f_B, f_{h}) \otimes (p_A, p_B, p_{h}) = (f_A\otimes p_A, f_B\otimes p_B, \varphi_0 \circ (f_{h}\otimes p_{h}) \circ {\varphi_1})$ where $\varphi_0$ and $\varphi_1$ are appropriate coherence maps associated to $\otimes \colon \B\times\B \to \B$.

The monoidal bicategory $(\B \Downarrow \B)$ is defined as follows.
Its objects are 2-cells $\sigma \colon h \Rightarrow k \colon A \to B$ in $\B$.
A morphism from $\sigma$ to $\sigma' \colon h' \Rightarrow k' \colon A' \to B'$ is a 4-tuple $(f_A, f_B, f_h, f_k)$ where $(f_A, f_B, f_h)$ and $(f_A, f_B, f_k)$ take the same form as~\eqref{comma_one_cell} and satisfy
\begin{equation*}
\vcenter{\hbox{
\begin{tikzpicture}[y=0.60pt, x=0.6pt,xscale=-1, inner sep=0pt, outer sep=0pt, every text node part/.style={font=
\scriptsize} ]
\path[draw=black,line join=miter,line cap=butt,line width=0.650pt]
  (210.0000,902.3622) .. controls (250.0000,902.3622) and (300.0000,872.3622) ..
  node[above left=0.12cm, at start] {$f_A$} (300.0000,872.3622);
\path[draw=black,line join=miter,line cap=butt,line width=0.650pt]
  (210.0000,842.3622) .. controls (250.0000,842.3622) and (300.0000,872.3622) ..
  node[above left=0.12cm, at start] {$k'$} (300.0000,872.3622);
\path[draw=black,line join=miter,line cap=butt,line width=0.650pt]
  (420.0000,902.3622) .. controls (390.0000,902.3622) and (360.0000,892.3622) ..
  node[above right=0.12cm, at start] {$h$} (360.0000,892.3622);
\path[draw=black,line join=miter,line cap=butt,line width=0.650pt]
  (301.7391,871.0578) .. controls (321.9565,852.1448) and (390.0000,841.4926) ..
  node[above right=0.12cm, at end] {$f_B$} (420.0000,842.5796);
\path[draw=black,line join=miter,line cap=butt,line width=0.650pt]
  (301.5217,873.2318) .. controls (314.1304,883.0144) and (358.2609,892.1448) ..
  node[above=0.12cm, pos=0.5] {$k$} (358.2609,892.1448);
\path[fill=black] (360,892.36218) node[circle, draw, line width=0.65pt, minimum
  width=5mm, fill=white, inner sep=0.25mm] (text3097) {$\sigma$    };
\path[fill=black] (300,872.36218) node[circle, draw, line width=0.65pt, minimum
  width=5mm, fill=white, inner sep=0.25mm] (text3101) {$f_k$    };
\end{tikzpicture}
}}
\quad \ = \ \quad
\vcenter{\hbox{
\begin{tikzpicture}[y=0.60pt, x=0.6pt,xscale=-1, inner sep=0pt, outer sep=0pt, every text node part/.style={font=
\scriptsize} ]
\path[draw=black,line join=miter,line cap=butt,line width=0.650pt]
  (341.7391,973.6665) .. controls (353.4783,981.4926) and (385.6522,1002.1448)
  .. node[above right=0.12cm, at end] {$h$} (419.5652,1002.3622);
\path[draw=black,line join=miter,line cap=butt,line width=0.650pt]
  (338.9130,973.4491) .. controls (323.6957,987.5796) and (257.1739,1002.3622)
  .. node[above left=0.12cm, at end] {$f_A$} (210.4348,1001.9274);
\path[draw=black,line join=miter,line cap=butt,line width=0.650pt]
  (342.1739,971.4926) .. controls (349.7826,964.3187) and (385.0000,943.0144) ..
  node[above right=0.12cm, at end] {$f_B$} (419.7826,942.5796);
\path[draw=black,line join=miter,line cap=butt,line width=0.650pt]
  (338.6957,970.4057) .. controls (325.0000,961.0578) and (290.8696,957.7970) ..
  node[above=0.12cm, pos=0.5] {$h'$} (272.3913,953.2318);
\path[draw=black,line join=miter,line cap=butt,line width=0.650pt]
  (267.8261,951.2752) .. controls (267.8261,951.2752) and (240.2174,943.4491) ..
  node[above left=0.12cm, at end] {$k'$} (210.4348,942.5796);
\path[fill=black] (340,972.36218) node[circle, draw, line width=0.65pt, minimum
  width=5mm, fill=white, inner sep=0.25mm] (text3240-1-9-2) {$f_h$    };
\path[fill=black] (270,952.36218) node[circle, draw, line width=0.65pt, minimum
  width=5mm, fill=white, inner sep=0.25mm] (text3057) {$\sigma'$    };
\end{tikzpicture}
}}\ .
\end{equation*}
A 2-cell $(f_A, f_B, f_h, f_k) \Rightarrow (g_A, g_B, g_h, g_k)$ is a pair $(\alpha_A, \alpha_B)$,
$\alpha_A \colon f_A \Rightarrow g_A$, $\alpha_B \colon f_B \Rightarrow g_B$ satisfying~\eqref{comma_two_cell} for both $h$ and $k$.
Composition and identities are defined in the obvious way.
Again, the tensor for the monoidal structure is defined on 0 and 2-cells by tensoring the underlying data in $\B$.
On 1-cells, we need $(f_A, f_B, f_h, f_k) \otimes (p_A, p_B, p_h, p_k) = (f_A\otimes p_A, f_B\otimes p_B, \varphi_0 \circ (f_{h}\otimes p_{h}) \circ {\varphi_1}, \varphi_2 \circ (f_{h}\otimes p_{h}) \circ {\varphi_3})$ where each $\varphi_i$ is an appropriate coherence map associated to $\otimes \colon \B\times\B \to \B$.

The internal bicategory structure
\begin{equation*}
 \cd[]{
  (\B \Downarrow \B) \ar@<-.5ex>[r] \ar@<.5ex>[r] &
  (\B \downarrow \B) \ar@<-.5ex>[r] \ar@<.5ex>[r] &
  (\B) }
 \end{equation*}
is given by first defining composition of 1-cells~\eqref{comma_one_cell} to be `down-the-page'.
Domain maps, codomain maps, identities and 2-cell composition follow easily from there.
\end{observation}

\begin{sidewaystable}
\centering
\begin{tabular}{| c || c | c | c |}
\hline
  dim. & $\N(\B)$ & $\N(\B \downarrow \B)$ & $\N(\B \Downarrow \B)$ \\
\hline
  0 & $\star$ & $\star$ & $\star$ \\
\hline
 1 & A &
  $\cd[]{
A \ar[d]_{p}  \\
A'
}$
 & $\cd[]{
A \ar@/_1em/[d]_{p} \ar@/^1em/[d]^{\hat{p}} \rtwocell{d}{\sigma}  \\
A'
}$ \\
\hline
  2 &
  $\cd[]{ A\otimes B \ar[r]^-{f} & C }$
  &
$\cd[@R+1em@C+1.5em]{
A \otimes B \ar[d]_{p \otimes q} \ar[r]^{f}
  & C \ar[d]^{r} \\
A' \otimes B' \ar[r]_{f'} \rtwocell{ur}{\zeta}
  & C'
}$
& $
\cd[@!C@R+2em@C+1.5em]{
 A B \ar[r]^{f} \ar@/_1.5em/[d]_{p q} \ar@/^1.5em/[d]^(0.55){\hat{p} \hat{q}} \rtwocell{d}{\sigma \tau}
   & \ar@/_1em/[d]_(0.3){r} \ar@/^1em/[d]^{\hat{r}} \rtwocell{d}{\mu} C \\
 A' B' \ar[r]_{f'}
   & C' \rtwocell[0.3]{ul}{\hat{\zeta}} \rtwocell[0.65]{ul}{\zeta}
}$ \\
\hline
  3  &
 $\cd[@R+1.5em@C-0.7em]{
    (A B) D \rtwocell{drr}{\gamma} \ar[rr]^{\mathfrak{a}} \ar[d]_{f D}
      && A (B D) \ar[d]^{A h} \\
    C D \ar[r]_{g}
      & E
      & A F \ar[l]^{k}
  }$
 & $\cd[@!C@C-.5em]{
 & A(BD) \congcell{dddl} \ar[rr]^{Ah} \ar[dd]_(0.3){p(qs)} \rtwocell{ddrr}{p \xi}
     && AF \ar[dd]^(0.3){pt} \ar[dr]^{k} \rtwocell[0.6]{dddr}{\phi}
     & \\
 (AB)D \ar@[gray][drr] \ar[ur]^{\mathfrak{a}} \ar[dd]_{(pq)s}
     &
     &&
     & \ar[dd]_{u} E \\
 \ar@[gray][d]  & A'(B'D') \ar[rr]_(0.75){B'h'} \utwocell[0.4]{ddrrr}{\gamma'}
     & \ar@[gray][urr] \ar@[gray][dd] \textcolor{gray}{CD}
     & A'F' \ar[dr]^{k'}
     & \\
 (A'B')D' \ar[ur]^{\mathfrak{a}} \ar[drr]_{f'D'}
     &&  && E' \\
 && C'D' \ar[urr]_{g'} &&
 }$ & \cellcolor{gray!25} \\
\hline
  4  & See~\eqref{NB4simplex} on p.\pageref{NB4simplex} & \cellcolor{gray!25} & \cellcolor{gray!25} \\
\hline
 $\ge 5$  & \cellcolor{gray!25} & \cellcolor{gray!25} & \cellcolor{gray!25} \\
\hline
\end{tabular}
\caption[0, 1 and 2-cells in the lowest few dimensions of the simplicial bicategory $\BN\B$]
{
0, 1 and 2-cells in the lowest few dimensions of the simplicial bicategory $\BN\B$.
Grey entries are uniquely determined by coskeletality conditions, see Observation~\ref{obs:4-coskeletal}.
The 2-cells on the front face of the pentagonal prism have been omitted; they can easily be filled in by the reader.
Unlabelled isomorphisms are composites of basic coherence data in $\B$.
The 2-cell $p\xi$ is schematic for the obvious pasting of $\xi$ with $p$ and coherence isomorphisms.
}
\label{table1}
\end{sidewaystable}

\begin{observation}
For any simplicial set $X$, $\mathrm{sSet}(X, \BN\B)$ is a bicategory.
To see this, note that the representable $\mathrm{sSet}(X, \mhyphen) \colon \mathrm{sSet} \to \mathrm{Set}$ preserves limits.
Now if $\BY$ is a bicategory
\begin{equation*}
 \cd[]{
  Y_2 \ar@<-.5ex>[r] \ar@<.5ex>[r] &
  Y_1 \ar@<-.5ex>[r] \ar@<.5ex>[r] &
  Y_0 }
 \end{equation*}
internal to $\mathrm{sSet}$, the 2-globular set
\begin{equation*}
 \cd[]{
  \mathrm{sSet}(X,Y_2) \ar@<-.5ex>[r] \ar@<.5ex>[r] &
  \mathrm{sSet}(X,Y_1) \ar@<-.5ex>[r] \ar@<.5ex>[r] &
  \mathrm{sSet}(X,Y_0) }
 \end{equation*}
is a bicategory which we call $\mathrm{sSet}(X,\BY)$.
\end{observation}

The following observation is useful for understanding the nature of 0,1 and 2-cells in 
$\mathrm{sSet}(X, \BN\B)$.

\begin{observation} \label{obs:4-coskeletal}
Since $\N(\B)$ is 4-coskeletal, its simplices at dimension 5 and above are uniquely determined by their boundary.

This is also true for $\N(\B\downarrow\B)$, but it has a stronger property: a $4$-simplex in $\N(\B\downarrow\B)$ is uniquely determined by its boundary 3-simplices and its source and target 4-simplices in $\N(\B)$.

The simplicial set $\N(\B\Downarrow\B)$ has both of these properties and an even stronger one: each 3-simplex is uniquely determined by its boundary $2$-simplices and its source and target 3-simplices in $\N(\B\downarrow\B)$.
This means that the essential data of these simplicial sets are contained in their lowest $4,3,2$ dimensions respectively.
In particular this means that a 0-cell in $\mathrm{sSet}(X,\BN\B)$, a map $X \to \N(\B)$, is completely determined by its behaviour up to dimension $4$.
A 1-cell in $\mathrm{sSet}(X,\BN\B)$, a map $X \to \N(\B\downarrow\B)$, is completely determined by its behaviour up to dimension $3$ and its source and target.
And a 2-cell in $\mathrm{sSet}(X,\BN\B)$, a map $X \to \N(\B\Downarrow\B)$, is completely determined by its behaviour up to dimension $2$ and its source and target.

In order to be more rigorous we make the following definition.
Let $F \colon X \to Y$ be a map of simplicial sets.
We say that \emph{$F$ is $m$-coskeletal} or that \emph{$X$ is $m$-coskeletal over $Y$} when $F$ has the unique right-lifting property with respect to boundary inclusions $\delta\Delta_n \to \Delta_n$ for all $n>m$.
That is, for all $u,v$ as below there exists a unique $k$ making the both triangles commute.
\begin{equation*}
\cd[]
{
  \partial\Delta_n  \ar@{^{(}->}[d] \ar[r]^{u} & X \ar[d]^{F} \\
  \Delta_n \ar@{-->}[ur]|{k} \ar[r]_{v} & Y
}\ .
\end{equation*}
A simplicial set $X$ is $m$-coskeletal precisely when it is $m$-coskeletal over $\mathbbm{1}$.
We can now restate the observation as:
\begin{itemize}
\item $\N(\B)$ is 4-coskeletal;
\item $\N(\B\downarrow\B)$ is 3-coskeletal over $\N\B \times \N\B$ via $(\N s,\N t)$; and
\item $\N(\B\Downarrow\B)$ is 2-coskeletal over $\N(\B\downarrow\B) \times_{(\N s,\N t)} \N(\B\downarrow\B)$ via $(\N s, \N t)$.
\end{itemize}

To justify our observation, note the following.
The nerve functor $\N \colon \mathrm{MonBicat}_s \to \mathrm{sSet}$ has the property that it sends every locally faithful functor to a 3-coskeletal map and every locally fully faithful functor to a 2-coskeletal map.
The maps
$ (s, t) \colon  (\B \downarrow \B) \to \B \times \B$ and
$ (s, t) \colon  (\B \Downarrow \B) \to (\B \downarrow \B) \times_{(s,t)} (\B \downarrow \B)$
are locally faithful and locally fully faithful respectively.
\end{observation}

\section{Classifying skew monoidales}\label{sec:direct_biequivalence}

In this section we show that simplicial maps from $\C$ to $\N\B$ are skew monoidales in $\B$ in the sense that there is a biequivalence
$$ \mathrm{sSet}(\C, \BN\B) \simeq \mathrm{SkMon}(\B)\ .$$
Before we formally construct this biequivalence let us examine the data of a simplicial map $F \colon \C \to \N \B$ and highlight the difficulties that arise.
The data for such an $F$ consist of the following:
\begin{itemize}
\item A single object $A$ in $\B$
\item Two 1-cells in $\B$
\begin{equation*}
\cd[]{A\otimes A \ar[r]^-{t} & A}\quad\text{and}\quad\cd[]{I\otimes I \ar[r]^-{i} & A}
\end{equation*}
\item Four 2-cells
\begin{equation*}
 \cd{
    (A \otimes A) \otimes A \rtwocell{drr}{a} \ar[rr]^{\mathfrak{a}} \ar[d]_{t \otimes A} & &
    A \otimes (A \otimes A) \ar[d]^{A \otimes t} \\
    A \otimes A \ar[r]_{t} & A & A
    \otimes A \ar[l]^{t}
  }
\end{equation*}
\begin{equation*}
 \cd{
    (A \otimes I) \otimes I \rtwocell{drr}{r} \ar[rr]^{\mathfrak{a}} \ar[d]_{\mathfrak{r}^\centerdot \otimes 1} & &
    A \otimes (I \otimes I) \ar[d]^{A \otimes i} \\
    A \otimes I \ar[r]_{\mathfrak{r}^\centerdot} & A & A
    \otimes A \ar[l]^{t}
  }
\end{equation*}
\begin{equation*}
 \cd{
    (I \otimes I) \otimes A \rtwocell{drr}{\ell} \ar[rr]^{\mathfrak{a}} \ar[d]_{i \otimes 1} & &
    I \otimes (I \otimes A) \ar[d]^{1 \otimes \mathfrak{l}} \\
    A \otimes A \ar[r]_{t} & A & I
    \otimes A \ar[l]^{\mathfrak{l}}
  }
\end{equation*}
\begin{equation*}\label{k}
 \cd{
    (I \otimes I) \otimes I \rtwocell{drr}{k} \ar[rr]^{\mathfrak{a}} \ar[d]_{i \otimes 1} & &
    I \otimes (I \otimes I) \ar[d]^{1 \otimes i} \\
    A \otimes I \ar[r]_{\mathfrak{r}^\centerdot} & A & I
    \otimes A \ar[l]^{\mathfrak{l}}
  }
\end{equation*}
\item And nine equalities
\end{itemize}
\begin{equation}\label{stringaxiom1}
        \def\abc{t}\def\abd{t}\def\abe{t}
        \def\acd{t}\def\ace{t}\def\ade{t}
        \def\bcd{t}\def\bce{t}
        \def\bde{t}
        \def\cde{t}
        \def\abde{a}\def\acde{a}
        \def\abcd{a}\def\abce{a}
        \def\bcde{a}
        \stringpent
\end{equation}
\begin{equation}\label{stringaxiom2}
        \def\abc{\mathfrak{l}}\def\abd{t}\def\abe{t}
        \def\acd{\mathfrak{l}}\def\ace{t}\def\ade{t}
        \def\bcd{i}\def\bce{\mathfrak{r}^\centerdot}
        \def\bde{t}
        \def\cde{\mathfrak{r}^\centerdot}
        \def\abde{a}\def\acde{s_1t}
        \def\abcd{\ell}\def\abce{s_1t}
        \def\bcde{r}
        \stringpent
\end{equation}
\begin{equation}\label{stringaxiom3}
        \def\cde{t}\def\bde{\mathfrak{r}^\centerdot}\def\ade{t}
        \def\bce{\mathfrak{r}^\centerdot}\def\ace{t}\def\abe{\mathfrak{r}^\centerdot}
        \def\bcd{\mathfrak{r}^\centerdot}\def\acd{t}
        \def\abd{\mathfrak{r}^\centerdot}
        \def\abc{i}
        \def\abde{s_2t}\def\abce{r}
        \def\bcde{s_2t.}\def\acde{a}
        \def\abcd{r}
        \stringpent
\end{equation}
\begin{equation}\label{stringaxiom4}
        \def\cde{i}\def\bde{\mathfrak{l}}\def\ade{\mathfrak{l}}
        \def\bce{t}\def\ace{t}\def\abe{t}
        \def\bcd{\mathfrak{l}}\def\acd{\mathfrak{l}}
        \def\abd{t}
        \def\abc{t}
        \def\abde{s_0t}\def\abce{a}
        \def\bcde{\ell}\def\acde{\ell}
        \def\abcd{s_0t}
        \stringpent
\end{equation}
\begin{equation}\label{stringaxiom5}
        \def\cde{\mathfrak{l}}\def\bde{i}\def\ade{\mathfrak{l}}
        \def\bce{i}\def\ace{i}\def\abe{\mathfrak{l}}
        \def\bcd{\mathfrak{l}}\def\acd{i}
        \def\abd{i}
        \def\abc{\mathfrak{l}}
        \def\abde{k}\def\abce{s_2i}
        \def\bcde{s_1i.}\def\acde{s_0i}
        \def\abcd{s_1i}
        \stringpent
\end{equation}
\begin{equation}\label{stringaxiom6}
        \def\cde{i}\def\bde{t}\def\ade{\mathfrak{l}}
        \def\bce{\mathfrak{r}^\centerdot}\def\ace{t}\def\abe{\mathfrak{r}^\centerdot}
        \def\bcd{\mathfrak{r}^\centerdot}\def\acd{\mathfrak{l}}
        \def\abd{i}
        \def\abc{i}
        \def\abde{k}\def\abce{r}
        \def\bcde{s_2i.}\def\acde{\ell}
        \def\abcd{s_0i}
        \stringpent
\end{equation}
\begin{equation}\label{stringaxiom7}
        \def\cde{i}\def\bde{\mathfrak{l}}\def\ade{\mathfrak{l}}
        \def\bce{\mathfrak{r}^\centerdot}\def\ace{t}\def\abe{\mathfrak{r}^\centerdot}
        \def\bcd{i}\def\acd{\mathfrak{l}}
        \def\abd{\mathfrak{r}^\centerdot}
        \def\abc{i}
        \def\abde{s_0s_1c}\def\abce{r}
        \def\bcde{k}\def\acde{\ell}
        \def\abcd{k}
        \stringpent
\end{equation}
\begin{equation}\label{stringaxiom8}
        \def\cde{i}\def\bde{\mathfrak{l}}\def\ade{\mathfrak{l}}
        \def\bce{\mathfrak{r}^\centerdot}\def\ace{t}\def\abe{t}
        \def\bcd{\mathfrak{l}}\def\acd{\mathfrak{l}}
        \def\abd{t}
        \def\abc{\mathfrak{l}}
        \def\abde{s_0t}\def\abce{s_1t}
        \def\bcde{k}\def\acde{\ell}
        \def\abcd{\ell}
        \stringpent
\end{equation}
\begin{equation}\label{stringaxiom9}
        \def\cde{\mathfrak{r}^\centerdot}\def\bde{t}\def\ade{t}
        \def\bce{\mathfrak{r}^\centerdot}\def\ace{t}\def\abe{\mathfrak{r}^\centerdot}
        \def\bcd{i}\def\acd{\mathfrak{l}}
        \def\abd{\mathfrak{r}^\centerdot}
        \def\abc{i}
        \def\abde{s_2t}\def\abce{r}
        \def\bcde{r}\def\acde{s_1t}
        \def\abcd{k}
        \stringpent
\end{equation}
The similarity with skew monoidales in $\B$ is strong but there are some problems.

As previously mentioned, the unit map for a skew monoidale is of the form $I \to A$ but $i$ is a map $I\otimes I
\to A$.
Similarly, the left and right unit constraints for a skew monoidale have different domains and codomains than the $r$ and $\ell$ shown here.
These differences amount to the fact that $I\otimes I$ does not equal $I$; we will deal with this momentarily.
The second problem is that there is an extra coherence 2-cell $k$.
Fortunately, the equality in~\eqref{stringaxiom5} together with the monoidal bicategory axioms force $k$ to be equal to the pasting
\begin{equation} \label{kappastring}
  \vcenter{\hbox{\begin{tikzpicture}[y=0.8pt, x=0.9pt,yscale=-1, inner
        sep=0pt, outer sep=0pt, every text node
        part/.style={font=\scriptsize} ]
  \path[draw=black,line join=miter,line cap=butt,line width=0.650pt]
  (60.5482,860.1807) .. controls (97.7329,834.6731) and (20.2790,833.4806) ..
  node[above=0.07cm,pos=0.5] {$\mathfrak{r}^\centerdot 1$}(21.6649,878.0631) .. controls (22.2692,897.5009) and
(50.8157,910.5973) .. (77.9793,909.6016)
  (83.8468,909.1607) .. controls (91.1546,908.3220) and (98.2018,906.3945) ..
  node[above left=0.055cm,pos=0.55] {$\mathfrak{r}$}(104.3947,903.2213)
  (10.0000,912.3622) .. controls (28.6404,912.3622) and (47.9703,914.2230) ..
  node[above right=0.12cm,at start] {\!\!$i.1$}(68.7323,916.3393)
  (77.0089,917.1848) .. controls (95.0409,919.0225) and (114.1847,920.8922) ..
  node[below=0.074cm,pos=0.5] {$i$}(134.9003,921.7992)
  (140.5333,922.0201) .. controls (146.8711,922.2390) and (153.3558,922.3622) ..
  node[above left=0.12cm,at end] {$1.i$\!\!}(160.0000,922.3622)
  (10.0000,932.3622) .. controls (46.3655,932.3622) and (117.0711,908.0690) ..
  node[above right=0.12cm,at start] {\!$\mathfrak{r}^\centerdot$}(60.5546,892.8673)
  (61.6193,889.4751) .. controls (137.9573,830.7177) and (122.5843,941.9527) ..
  node[above left=0.12cm,at end] {$\mathfrak{l}$\!}(160.0000,942.3622)
  (59.5990,867.2517) .. controls (61.2888,870.8982) and (66.6120,874.7013) ..
  node[above right=0.05cm,pos=0.5,rotate=-35]{$1 \mathfrak{l}$}(73.0425,878.5617)
  (78.0172,881.4569) .. controls (88.8715,887.6317) and (100.9341,893.8888) ..
  node[above right=0.035cm,pos=0.4]{$\mathfrak{l}$}(104.0457,899.8292);
  \path[draw=black,line join=miter,line cap=butt,line width=0.650pt]
  (62.3769,864.7264) .. controls (128.2265,830.2570) and (130.2744,900.0073) ..
  node[above left=0.12cm,at end] {$\mathfrak{a}$\!}(160.0000,902.3622);
  \path[fill=black] (52.527935,864.72632) node[circle, draw, line width=0.65pt, minimum width=5mm, fill=white, inner
sep=0.25mm] (text3790) {$\mu^{\scriptscriptstyle -1}$  };
  \path[fill=black] (55.101635,890.83112) node[circle, draw, line width=0.65pt, minimum width=5mm, fill=white, inner
sep=0.25mm] (text3790-1) {$\theta^{\scriptscriptstyle -1}$   };
  \path[fill=black] (104.09403,901.18524) node[circle, draw, line width=0.65pt, minimum width=5mm, fill=white, inner
sep=0.25mm] (text3790-6) {$\theta$};
      \end{tikzpicture}}}
\end{equation}
and thus completely specified by the coherence data of $\B$.
The third problem is that there are too many axioms!
Fortunately, \eqref{stringaxiom8} and~\eqref{stringaxiom9} hold trivially in any monoidal bicategory.
The remaining six equalities are precisely the five axioms we require (axioms~\eqref{stringaxiom6} and~\eqref{stringaxiom7} are equivalent).

We have yet to resolve the first problem: $i, \ell, r$ have the wrong shape.
This is resolved by constructing a new monoidal bicategory $\B^*$ which is like $\B$, but with unit object $I \otimes I$ and appropriately modified coherence data.
Then simplicial maps from $\C$ to $\N\B$ are exactly skew monoidales in $\B^*$: there is an isomorphism of bicategories
$$ \mathrm{sSet}(\C, \BN\B) \cong \mathrm{SkMon}(\B^*)\ .$$
Then since $\B^* \simeq \B$ we know that $\mathrm{SkMon}(\B^*) \simeq \mathrm{SkMon}(\B)$ and we have the required correspondence.

\begin{definition}
Suppose $\B$ is a monoidal bicategory $\B = (\B, \otimes, I, \mathfrak{a}, \mathfrak{l}, \mathfrak{r},\pi, \mu, \sigma, \tau)$.
Let $\B^*$ be the monoidal bicategory $(\B, \otimes, I\otimes I, \mathfrak{a}, \mathfrak{l}^*, \mathfrak{r}^*, \pi, \mu^*, \sigma^*, \tau^*)$ where the new data are defined as follows.
The new pseudo-natural tranformations $\mathfrak{l}^*$ and $\mathfrak{r}^*$ have 1-cell components
\begin{equation*}
 \mathfrak{l}^*_{A} = \cd[@!C]{(I\otimes I)\otimes A \ar[r]^-{\mathfrak{l}\otimes A} & I\otimes A \ar[r]^-{\mathfrak{l}} & A}
\end{equation*}
  and
\begin{equation*}
 \mathfrak{r}^*_{A} = \cd[@!C]{A \ar[r]^-{\mathfrak{r}} & A\otimes I \ar[r]^-{A\otimes \mathfrak{r}} & A\otimes (I\otimes I)}\ .
\end{equation*}
Their 2-cell components can easily be deduced.
The modifications $\mu^*$, $\sigma^*$ and $\tau^*$ have 2-cell components
\begin{equation*}
\begin{gathered}
\mu^*_{AB} =
 \cd[@!C@C-2.5em]
 {
    && (A(II))B \ar[rr]^{\mathfrak{a}} \congcell{drr} && A(II)B) \ar[dr]^{A(\mathfrak{l}B)} && \\
    & A(IB) \ar[ur]^{(A\mathfrak{r})B} \ar[rrrr]^{\mathfrak{a}} && \dtwocell{d}{\mu} && (AI)B \ar[dr]^{A\mathfrak{l}} & \\
    AB \ar[ur]^{\mathfrak{r}B} \ar[rrrrrr]_{id} &&&&&& AB
 }
\end{gathered}
\end{equation*}
\begin{equation*}
\begin{gathered}
 \sigma^*_{BC} =
 \cd[@!C@C-2.5em]
 {
    & (II)(BC) \congcell{dr} \ar[rr]^{\mathfrak{l}(BC)} && I(BC) \dtwocell{d}{\sigma} \ar[dr]^{\mathfrak{l}} & \\
    ((II)B)C \ar[ur]^{\mathfrak{a}} \ar[rr]_{(\mathfrak{l}B)C} && (IB)C \ar[ur]^{\mathfrak{a}} \ar[rr]_{\mathfrak{l}C} && BC
 }  \\
 \tau^*_{AB} =
 \cd[@!C@C-2em]
 {
 	 & (AB)I \dtwocell{d}{\tau} \ar[dr]^{\mathfrak{a}} \ar[rr]^{(AB)\mathfrak{r}} && (AB)(II) \ar[dr]^{\mathfrak{a}} \\
    AB \ar[ur]^{\mathfrak{r}} \ar[rr]_{A\mathfrak{r}} && A(BI) \congcell{ur} \ar[rr]_{A(B\mathfrak{r})} && A(B(II))
 }
\end{gathered}
\end{equation*}
Unlabelled isomorphisms come from pseudo-naturality of $\mathfrak{a}$.
It is not hard to check that this data satisfies the required axioms and that $\B \simeq \B^*$ as monoidal bicategories.
\end{definition}

\begin{theorem} \label{thm:main1}
For all monoidal bicategories $\B$ there is an isomorphism of bicategories $$ \mathrm{sSet}(\C, \BN\B) \cong \mathrm{SkMon}(\B^*)\ .$$
\end{theorem}
\begin{proof}

Suppose that $F \colon \C \to \N\B$.
Let us compare the image of $F$ with the data for a skew monoidale in $\B^*$ and demonstrate a correspondence between the two.
At dimensions one and two these data are exactly equal: a single object $A$, a tensor map $t$ and a unit map $i \colon I\otimes I \to A$.
At dimension two, the $2$-cell $a$ has the same form as the associativity constraint $\alpha$ for a skew-monoidale; whilst, as observed above, $k$ is necessarily of the form~\eqref{kappastring}.
On the other hand, the data $\ell$ and $r$ give rise to left and right unit constraints $\lambda$ and $\rho$ upon forming the composites
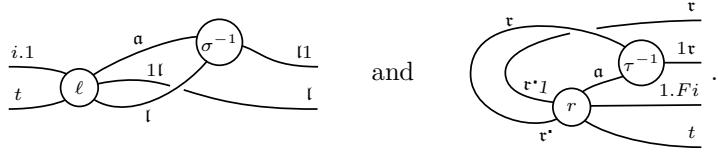
\begin{equation*}
  \vcenter{\hbox{\begin{tikzpicture}[y=0.8pt, x=0.9pt,yscale=-1, inner
        sep=0pt, outer sep=0pt, every text node
        part/.style={font=\scriptsize} ]
  \path[draw=black,line join=miter,line cap=butt,line width=0.650pt]
  (77.4746,873.6249) .. controls (65.7563,893.8888) and (37.6777,911.3520) ..
  node[below right=0.06cm,pos=0.6] {$\mathfrak{l}$}(22.5254,894.9370)
  (119.7970,902.3622) .. controls (89.8639,902.6202) and (75.3129,897.3809) ..
  node[above left=0.12cm,at start] {$\mathfrak{l}$\!}(63.5684,893.3012)
  (57.8098,891.3693) .. controls (48.0652,888.3122) and (38.8450,886.9959) ..
  node[above right=0.06cm,pos=0.4] {$1\mathfrak{l}$}(21.7678,891.8571);
  \path[draw=black,line join=miter,line cap=butt,line width=0.650pt]
  (77.5761,869.1198) .. controls (65.5753,863.5577) and (30.6214,880.5944) ..
  node[above left=0.06cm,pos=0.5] {$\mathfrak{a}$}(22.2728,888.9282);
  \path[draw=black,line join=miter,line cap=butt,line width=0.650pt]
  (120.0495,882.3622) .. controls (94.0270,882.3622) and (96.0094,871.2505) ..
  node[above left=0.12cm,at start] {$\mathfrak{l} 1$\!}(80.0495,871.2505);
  \path[draw=black,line join=miter,line cap=butt,line width=0.650pt]
  (-10.0000,882.3622) .. controls (-0.5398,882.3622) and (13.5355,882.8673) ..
  node[above right=0.12cm,at start] {\!\!$i. 1$}(19.0914,888.9282);
  \path[draw=black,line join=miter,line cap=butt,line width=0.650pt]
  (-10.0000,902.3622) .. controls (-0.8054,902.3622) and (14.5457,900.5944) ..
  node[above right=0.12cm,at start] {\!$t$}(18.8388,894.0284);
  \path[fill=black] (78.568542,870.78729) node[circle, draw, line width=0.65pt, minimum width=5mm, fill=white, inner
sep=0.25mm] (text3387) {$\sigma^{\scriptscriptstyle -1}$     };
  \path[fill=black] (19.697973,892.25299) node[circle, draw, line width=0.65pt, minimum width=5mm, fill=white, inner
sep=0.25mm] (text3391) {$\ell$   };
      \end{tikzpicture}}} \qquad
\text{and} \qquad
  \vcenter{\hbox{\begin{tikzpicture}[y=0.8pt, x=0.9pt,yscale=-1, inner
        sep=0pt, outer sep=0pt, every text node
        part/.style={font=\scriptsize} ]
  \path[draw=black,line join=miter,line cap=butt,line width=0.650pt]
  (106.0000,862.3622) -- node[above left=0.12cm,at end] {$1 \mathfrak{r}$\!}(130.0000,862.3622);
  \path[draw=black,line join=miter,line cap=butt,line width=0.650pt]
  (77.5241,882.8673) -- node[above left=0.12cm,at end] {$1.Fi$\!\!} (130.0000,882.3622);
  \path[draw=black,line join=miter,line cap=butt,line width=0.650pt]
  (77.5241,886.7048) .. controls (85.1003,895.0386) and (105.7563,902.3622) ..
  node[above left=0.12cm,at end] {$t$\!}(130.0000,902.3622);
  \path[draw=black,line join=miter,line cap=butt,line width=0.650pt]
  (103.9848,865.4840) .. controls (99.6916,869.7771) and (81.5698,871.5449) ..
  node[above left=0.05cm,pos=0.5] {$\mathfrak{a}$}(77.7817,878.8685);
  \path[draw=black,line join=miter,line cap=butt,line width=0.650pt]
  (103.7323,860.6857) .. controls (93.8833,844.1683) and (32.0387,832.8929) ..
  node[above=0.05cm,pos=0.66] {$\mathfrak{r}$}(33.0291,868.4666) .. controls (33.5742,888.0447) and (62.3769,897.3038) ..
  node[below=0.05cm,pos=0.8] {$\mathfrak{r}^\centerdot$}(73.2361,886.1921)
  (72.7259,882.0107) .. controls (66.4734,880.4919) and (47.8665,880.0812) ..
  node[above right=0.05cm,pos=0.66,rotate=-10] {$\mathfrak{r}^\centerdot 1$}(47.6192,867.2444) .. controls (47.4461,858.2562)
and (59.4891,852.3199) .. (74.8220,848.5084)
  (85.9387,846.1826) .. controls (102.4316,843.2893) and (120.1558,842.3313) ..
  node[above left=0.12cm,at end] {$\mathfrak{r}$\!}(130.0000,842.3622);
  \path[fill=black] (75.508904,883.41418) node[circle, draw, line width=0.65pt, minimum width=5mm, fill=white, inner
sep=0.25mm] (text3622) {$r$    };
  \path[fill=black] (104.80333,862.45349) node[circle, draw, line width=0.65pt, minimum width=5mm, fill=white, inner
sep=0.25mm] (text3626) {$\tau^{\scriptscriptstyle -1}$   };
      \end{tikzpicture}}}\ .
\end{equation*}
The assignments $\ell \mapsto \lambda$ and $r \mapsto \rho$ are in fact bijective, the former since it is given by composing with an invertible $2$-cell, and the latter since it is given by composition with an invertible $2$-cell followed by transposition under adjunction.
Thus the 2-dimensional data of $F$ and of a skew monoidale in $\B^*$ are in bijective correspondence.

Finally, after some calculation we find that, with respect to the $\alpha$, $\lambda$ and $\rho$ defined above, equations
\eqref{stringaxiom1},~\eqref{stringaxiom2},~\eqref{stringaxiom3},
\eqref{stringaxiom4},~\eqref{stringaxiom6} and~\eqref{stringaxiom7}
express precisely the five axioms for a skew monoidale in $\B$;
equation~\eqref{stringaxiom5} specifies $Fk$ and nothing more;
whilst equations~\eqref{stringaxiom8} and~\eqref{stringaxiom9} are both equalities which follow using only the axioms for a monoidal bicategory.
Thus, every simplicial map $F \colon \C \to \N\B$ determines a skew monoidale in $\B^*$ and this assignment is bijective.

Suppose next that $\gamma \colon F \to G$ is a 1-cell in $\mathrm{sSet}(\C, \N\B)$.
By Remark~\ref{obs:4-coskeletal}, $\gamma$ is determined by $F$ and $G$ and the data up to dimension 3 of a simplicial map $\gamma$ satisfying
\begin{equation*}
\cd[@!C@C+1.5em]{
 & \N(\B\downarrow\B) \ar[d]^{\N(s,t)} \\
 \C \ar[ur]^{\gamma} \ar[r]_-{(F,G)} & \N(\B) \times \N(\B)
 }\ .
\end{equation*}
This consists of:
\begin{itemize}
\item A single arrow $\gamma_c \colon Fc \to Gc$.
\item Two 2-cells
\begin{equation*}
\cd[]{
A \otimes A \ar[r]^-{Ft} \ar[d]_{\gamma_c \otimes \gamma_c} & A \ar[d]^{\gamma_c} \\
B \otimes B \ar[r]_-{Gt} \rtwocell{ur}{\gamma_{t}} & B
}
\quad \text{and} \quad
\cd[]{
I \otimes I \ar[r]^-{Fi} \ar[d]_{1 \otimes 1} & A \ar[d]^{\gamma_c} \\
I \otimes I \ar[r]_-{Gi} \rtwocell{ur}{\gamma_{i}} & B
}
\end{equation*}
where $A = Fc$ and $B = Gc$.
\item Four equations
\begin{equation} \label{transaxiom1}
\vcenter{\hbox{
}}
\end{equation}
where unlabelled isomorphisms come from pseudo-naturality of $\mathfrak{l}$ and $\mathfrak{r}^\centerdot$.
Displaying these isomorphisms explicitly (rather than using string-crossings) highlights the uniformity of the axioms.
\end{itemize}

We now compare the data for $\gamma$ with the data for a lax monoidal morphism in $\B^*$.
At dimension one these data are exactly equal: a single arrow $\gamma_c \colon A \to B$.
At dimension two, the 2-cell $\gamma_t$ has the same form as the tensor constraint $\phi$ for a lax monoidal morphism; on the other hand, $\gamma_i$ gives rise to unit constraint $\psi$ upon forming the composite
\begin{equation*}
\cd[@C+0.5cm]{
I \otimes I \ar@/_12pt/[d]_{1} \ar[r]^-{Fi} \ar@/^12pt/[d]|{1 \otimes 1} \ar@{}[d]|{\cong} & A
\ar[d]^{\gamma_c} \\
I \otimes I \ar[r]_-{Gi} & \rtwocell[0.35]{ul}{\gamma_{i}}  B
}
\ .
\end{equation*}
The assignment $\gamma_i \mapsto \psi$ is bijective since it is just precomposition with an invertible 2-cell.

Finally, after some calculation we find that, with respect to the $\phi$ and $\psi$ defined above, equations
\eqref{transaxiom1},~\eqref{transaxiom2} and~\eqref{transaxiom3}
express precisely the three axioms for a lax monoidal morphism in $\B$;
equation~\eqref{transaxiom4} is an equality which follows using only the axioms for a monoidal bicategory.
Thus, every 1-cell $\gamma \colon F \rightarrow G$ in $\mathrm{sSet}(\C, \BN\B)$ determines a lax monoidal morphism in $
\B^*$ and this assignment is bijective.

Suppose finally that $\Gamma \colon \gamma \Rightarrow \delta$ is a 2-cell in $\mathrm{sSet}(\C, \N\B)$.
By Remark~\ref{obs:4-coskeletal}, $\Gamma$ is determined by $\gamma$ and $\delta$ and the data up to dimension 2 of a simplicial map $\Gamma$ satisfying
\begin{equation*}
\cd[@!C@C+1.5em]{
 & \N(\B\Downarrow\B) \ar[d]^{\N(s,t)} \\
 \C \ar[ur]^{\Gamma} \ar[r]_-{(\gamma,\delta)} & \N(\B\downarrow\B) \times_{\N(s,t)} \N(\B\downarrow\B)
 }\ .
\end{equation*}
This consists of:
\begin{itemize}
\item A single 2-cell $\Gamma_c \colon \gamma_c \Rightarrow \delta_c$.
\item Two equations
\begin{equation} \label{modaxiom1}
\vcenter{\hbox{
\begin{tikzpicture}[y=0.60pt, x=0.6pt,yscale=-1, inner sep=0pt, outer sep=0pt, every text node part/.style={font=
\scriptsize} ]
\path[draw=black,line join=miter,line cap=butt,line width=0.650pt]
  (210.0000,902.3622) .. controls (250.0000,902.3622) and (300.0000,872.3622) ..
  node[above right=0.12cm, at start] {$Ft$} (300.0000,872.3622);
\path[draw=black,line join=miter,line cap=butt,line width=0.650pt]
  (210.0000,842.3622) .. controls (250.0000,842.3622) and (300.0000,872.3622) ..
  node[above right=0.12cm, at start] {$\gamma_c \gamma_c$} (300.0000,872.3622);
\path[draw=black,line join=miter,line cap=butt,line width=0.650pt]
  (420.0000,902.3622) .. controls (390.0000,902.3622) and (360.0000,892.3622) ..
  node[above left=0.12cm, at start] {$\delta_c$} (360.0000,892.3622);
\path[draw=black,line join=miter,line cap=butt,line width=0.650pt]
  (301.7391,871.0578) .. controls (321.9565,852.1448) and (390.0000,841.4926) ..
  node[above left=0.12cm, at end] {$Gt$} (420.0000,842.5796);
\path[draw=black,line join=miter,line cap=butt,line width=0.650pt]
  (301.5217,873.2318) .. controls (314.1304,883.0144) and (358.2609,892.1448) ..
  node[above, pos=0.5] {$\gamma_c$} (358.2609,892.1448);
\path[fill=black] (360,892.36218) node[circle, draw, line width=0.65pt, minimum
  width=5mm, fill=white, inner sep=0.25mm] (text3097) {$\Gamma_c$    };
\path[fill=black] (300,872.36218) node[circle, draw, line width=0.65pt, minimum
  width=5mm, fill=white, inner sep=0.25mm] (text3101) {$\gamma_t$    };
\end{tikzpicture}
}}
\quad \ = \ \quad
\vcenter{\hbox{
\begin{tikzpicture}[y=0.60pt, x=0.6pt,yscale=-1, inner sep=0pt, outer sep=0pt, every text node part/.style={font=
\scriptsize} ]
\path[draw=black,line join=miter,line cap=butt,line width=0.650pt]
  (341.7391,973.6665) .. controls (353.4783,981.4926) and (385.6522,1002.1448)
  .. node[above left=0.12cm, at end] {$\delta_c$} (419.5652,1002.3622);
\path[draw=black,line join=miter,line cap=butt,line width=0.650pt]
  (338.9130,973.4491) .. controls (323.6957,987.5796) and (257.1739,1002.3622)
  .. node[above right=0.12cm, at end] {$Ft$} (210.4348,1001.9274);
\path[draw=black,line join=miter,line cap=butt,line width=0.650pt]
  (342.1739,971.4926) .. controls (349.7826,964.3187) and (385.0000,943.0144) ..
  node[above left=0.12cm, at end] {$Gt$} (419.7826,942.5796);
\path[draw=black,line join=miter,line cap=butt,line width=0.650pt]
  (338.6957,970.4057) .. controls (325.0000,961.0578) and (290.8696,957.7970) ..
  node[above, pos=0.5] {$\delta_c \delta_c$} (272.3913,953.2318);
\path[draw=black,line join=miter,line cap=butt,line width=0.650pt]
  (267.8261,951.2752) .. controls (267.8261,951.2752) and (240.2174,943.4491) ..
  node[above right=0.12cm, at end] {$\gamma_c \gamma_c$} (210.4348,942.5796);
\path[fill=black] (340,972.36218) node[circle, draw, line width=0.65pt, minimum
  width=5mm, fill=white, inner sep=0.25mm] (text3240-1-9-2) {$\delta_t$    };
\path[fill=black] (270,952.36218) node[circle, draw, line width=0.65pt, minimum
  width=5mm, fill=white, inner sep=0.25mm] (text3057) {$\Gamma_c\Gamma_c$    };
\end{tikzpicture}
}}
\end{equation}
\begin{equation} \label{modaxiom2}
\vcenter{\hbox{
\begin{tikzpicture}[y=0.60pt, x=0.6pt,yscale=-1, inner sep=0pt, outer sep=0pt, every text node part/.style={font=
\scriptsize} ]
\path[draw=black,line join=miter,line cap=butt,line width=0.650pt]
  (210.0000,902.3622) .. controls (250.0000,902.3622) and (300.0000,872.3622) ..
  node[above right=0.12cm, at start] {$Fi$} (300.0000,872.3622);
\path[draw=black,line join=miter,line cap=butt,line width=0.650pt]
  (210.0000,842.3622) .. controls (250.0000,842.3622) and (300.0000,872.3622) ..
  node[above right=0.12cm, at start] {$1 1$} (300.0000,872.3622);
\path[draw=black,line join=miter,line cap=butt,line width=0.650pt]
  (420.0000,902.3622) .. controls (390.0000,902.3622) and (360.0000,892.3622) ..
  node[above left=0.12cm, at start] {$\gamma_c$} (360.0000,892.3622);
\path[draw=black,line join=miter,line cap=butt,line width=0.650pt]
  (301.7391,871.0578) .. controls (321.9565,852.1448) and (390.0000,841.4926) ..
  node[above left=0.12cm, at end] {$Gi$} (420.0000,842.5796);
\path[draw=black,line join=miter,line cap=butt,line width=0.650pt]
  (301.5217,873.2318) .. controls (314.1304,883.0144) and (358.2609,892.1448) ..
  node[above, pos=0.5] {$\gamma_c$} (358.2609,892.1448);
\path[fill=black] (360,892.36218) node[circle, draw, line width=0.65pt, minimum
  width=5mm, fill=white, inner sep=0.25mm] (text3097) {$\Gamma_c$    };
\path[fill=black] (300,872.36218) node[circle, draw, line width=0.65pt, minimum
  width=5mm, fill=white, inner sep=0.25mm] (text3101) {$\gamma_i$    };
\end{tikzpicture}
}}
\quad \ = \ \quad
\vcenter{\hbox{
\begin{tikzpicture}[y=0.60pt, x=0.6pt,yscale=-1, inner sep=0pt, outer sep=0pt, every text node part/.style={font=
\scriptsize} ]
\path[draw=black,line join=miter,line cap=butt,line width=0.650pt]
  (341.7391,973.6665) .. controls (353.4783,981.4926) and (385.6522,1002.1448)
  .. node[above left=0.12cm, at end] {$\gamma_c$} (419.5652,1002.3622);
\path[draw=black,line join=miter,line cap=butt,line width=0.650pt]
  (338.9130,973.4491) .. controls (323.6957,987.5796) and (257.1739,1002.3622)
  .. node[above right=0.12cm, at end] {$Fi$} (210.4348,1001.9274);
\path[draw=black,line join=miter,line cap=butt,line width=0.650pt]
  (342.1739,971.4926) .. controls (349.7826,964.3187) and (385.0000,943.0144) ..
  node[above left=0.12cm, at end] {$Gi$} (419.7826,942.5796);
\path[draw=black,line join=miter,line cap=butt,line width=0.650pt]
  (338.6957,970.4057) .. controls (325.0000,961.0578) and (290.8696,957.7970) ..
  node[above, pos=0.5] {$1 1$} (272.3913,953.2318);
\path[draw=black,line join=miter,line cap=butt,line width=0.650pt]
  (267.8261,951.2752) .. controls (267.8261,951.2752) and (240.2174,943.4491) ..
  node[above right=0.12cm, at end] {$11$} (210.4348,942.5796);
\path[fill=black] (340,972.36218) node[circle, draw, line width=0.65pt, minimum
  width=5mm, fill=white, inner sep=0.25mm] (text3240-1-9-2) {$\delta_i$    };
\path[fill=black] (270,952.36218) node[circle, draw, line width=0.65pt, minimum
  width=5mm, fill=white, inner sep=0.25mm] (text3057) {$1 1$    };
\end{tikzpicture}
}}\ .
\end{equation}
\end{itemize}

This is exactly the data of a monoidal transformation in $\B^*$: a single 2-cell satisfying exactly the required axioms. 
Thus, every 2-cell $\Gamma \colon \gamma \Rightarrow \delta$ in $\mathrm{sSet}(\C, \BN\B)$ determines a monoidal transformation in $\B^*$ and the assignment is bijective.

\end{proof}

The following result follows directly.

\begin{theorem} \label{thm:main2}
For all monoidal bicategories $\B$ there is a biequivalence $$ \mathrm{sSet}(\C, \BN\B) \simeq \mathrm{SkMon}(\B)\ .$$
\end{theorem}
\begin{proof}
From the biequivalence $\B^* \simeq \B$ we can show that $\mathrm{SkMon}(\B^*) \simeq \mathrm{SkMon}(\B)$. This, together with Theorem~\ref{thm:main1}, gives the desired result.
\end{proof}

\begin{remark}[Results for dual notions]
The biequivalence in Theorem \ref{thm:main2} applies to the bicategory of skew monoidales, lax monoidal morphisms and monoidal transformations.
The result is also true if we replace skew monoidales with opskew monoidales or ordinary monoidales.
We only need to change our definition of 3-simplices in $\N\B$ by reversing the direction of the 2-cells or making them invertible.
Similarly, the result holds if we replace lax monoidal morphisms with oplax monoidal morphisms or monoidal morphisms.
We only need to change our definition of 1-cells in $(\B\downarrow\B)$ by reversing the direction of the 2-cells or making them invertible.
\end{remark}

We conclude this section with some remarks on the connection between monoids and lax monoidal functors.
It well known that, for a monoidal category $\V$, there is an equivalence
$$ \mathrm{MonCat}_{\mathrm{lax}}(\mathbbm{1}, \V) \simeq \mathrm{Mon}(\V)$$ 
between lax monoidal functors $\mathbbm{1} \to \V$ and monoids internal to $\V$.
A similar result holds for skew monoidales internal to a monoidal bicategory.

\begin{definition}
Suppose that $\B$ and $\E$ are monoidal bicategories.
A \emph{lax monoidal homomorphism from $\B$ to $\E$} is a homomorphism  $F \colon \B \to \E$ on the underlying bicategories together with pseudo-natural families of maps
$$ \phi_{AB} \colon FA \otimes FB \to F(A \otimes B) \quad
\text{and} \quad \phi_I \colon I \to FI$$
and modifications $\omega, \gamma, \delta$ with (non-invertible) components
\begin{equation}
\cd[@!C@C+0.6cm]{
(FA \otimes FB) \otimes FC \ar[r]^{a} \ar[d]_{\phi \otimes FC} & FA \otimes (FB \otimes FC) \ar[d]^{FA \otimes \phi}\\
F(A \otimes B) \otimes FC \ar[d]_{\phi} & FA \otimes F(B \otimes C) \ar[d]^{\phi} \\
F((A \otimes B) \otimes C) \ar[r]_{Fa} & F(A \otimes (B \otimes C)) \rtwocell{uul}{\omega_{ABC}}
}
\end{equation}
\begin{equation}
\cd[@!C@C+0.2cm]{
I \otimes FA \ar[r]^{\phi \otimes FA} \ar@/_18pt/[ddr]_{l} & FI \otimes FA \ar[d]^{\phi} \\
 & F(IA) \ar[d]^{Fl} \\
 & FA \rtwocell{uul}{\gamma_A}
}
\end{equation}
\begin{equation}
\cd[@!C@C+0.2cm]{
FA \ar[r]^{r} \ar@/_18pt/[ddr]_{Fr} & FA \otimes I \ar[d]^{FA \otimes \phi} \\
 & FA \otimes FI \ar[d]^{\phi} \\
 & F(A \otimes I) \rtwocell{uul}{\delta_A}
}
\end{equation}
satisfying five axioms corresponding directly to those for skew-monoidal categories.
\end{definition}

By giving appropriate definitions of monoidal transformation and monoidal modification one can form a bicategory $\mathrm{MonBicat}_{\mathrm{lax}}(\B,\E)$ whose objects are lax monoidal homomorphisms.
We state the following without proof.

\begin{proposition}
For any monoidal bicategory $\B$, there is a biequivalence
$$ \mathrm{MonBicat}_{\mathrm{lax}}(\mathbbm{1}, \B) \simeq \mathrm{SkMon}(\B)\ .$$
\end{proposition}

This result is not unexpected and easy to verify, but we do need to take care that we have defined lax monoidal homomorphisms properly.
It is also relevant in light of the following remark.

\begin{remark} \label{intuition}
Suppose that $\V$ is a monoidal category. 
There is an equivalence
$$\mathrm{MonCat}_{\mathrm{lax}}(\mathbbm{1}, \V) \simeq \mathrm{Mon}(\V) $$
mentioned above, between lax monoidal functors from $\mathbbm{1}$ to $\V$ and monoids internal to $\V$. 
The data for a lax monoidal functor consists of a functor $F \colon \mathbbm{1} \to \V$ together with two natural families of maps $\phi_1 \colon 1 \to F1$ and $\phi_{11} \colon F1 \otimes F1 \to F1$ satisfying certain axioms. This is precisely an object $F1$ in $\V$ with a monoid structure and the correspondence is easily extended to an equivalence of categories.

There is a second equivalence
$$\mathrm{MonCat}_{\mathrm{nlax}}(\mathbbm{2}, \V) \simeq \mathrm{MonCat}_{\mathrm{lax}}(\mathbbm{1}, \V)$$
between normal lax functors from $\mathbbm{2}$ to $\V$ and lax functors from $\mathbbm{1}$ to $\V$. 
It exists as part of an adjunction where $\mathbbm{2}$ is the result of taking $\mathbbm{1}$ and freely adding a new unit object together with a map to the old unit object.

There is a third equivalence
$$ \mathrm{sSet}(\N\mathbbm{2}, \N\V) \simeq \mathrm{MonCat}_{\mathrm{nlax}}(\mathbbm{2}, \V) $$
obtained by observing that the nerve functor for monoidal categories is fully-faithful on normal lax functors.

Together, these form a sequence
\begin{equation}\label{seq_equiv_mon_cat}
 \mathrm{sSet}(\N\mathbbm{2}, \N\V) \simeq
  \mathrm{MonCat}_{\mathrm{nlax}}(\mathbbm{2}, \V) \simeq
  \mathrm{MonCat}_{\mathrm{lax}}(\mathbbm{1}, \V) \simeq
  \mathrm{Mon}(\V)
\end{equation}
Since $\mathbb{C}$ is isomorphic to $\N\mathbbm{2}$, this sequence demonstrates a correspondence between simplicial maps $\mathbb{C} \to \N\V$ and monoids internal to $\V$.

We fully expect that this sequence of equivalences can be generalised to the domain of monoidal bicategories and skew monoidales. 
Such a generalisation would require a suitable notion of normal lax functor for monoidal bicategories together with a proof that the nerve construction is essentially fully faithful on such functors.
That work would lead us too far from our current goal and so we leave it for another time.
\end{remark}

\section{Towards skew-monoidal bicategories}\label{sec:skew_monoidal_bicategories}

In this section we give a definition of skew-monoidal bicategory by looking at simplicial maps from $\C$ into a suitably defined nerve of $\mathrm{Bicat}$.
First, we describe a nerve of $\mathrm{Bicat}$ by informally regarding it as a monoidal tricategory, we then examine simplicial maps from $\C$ into this nerve. 
We find that a classification result for monoidal bicategories holds almost immediately, but the corresponding result for skew-monoidal bicategories requires an extra condition on the simplicial maps in question.
We then obtain a definition of skew-monoidal bicategory and find that skew-monoidal bicategories with invertible coherence data are precisely monoidal bicategories in the usual sense.
The data for a skew-monoidal bicategory consists of a single bicategory, tensor and unit maps, three coherence transformations, five coherence modifications and eight axioms.
This means one new coherence modification and five new axioms.

\subsection{A nerve of Bicat}

Let $\mathrm{Bicat}$ be the tricategory of bicategories, homomorphisms, pseudo-natural transformations and modifications.
Informally regarding it as a monoidal tricategory, we take the nerve of $\mathrm{Bicat}$ to be the simplicial set $\N(\mathrm{Bicat})$ defined as follows:
\begin{itemize}
\item There is a unique $0$-simplex $\star$.
\item A $1$-simplex is a bicategory $$\B_{01}\ ;$$ its two faces are necessarily $\star$.
\item A $2$-simplex is given by bicategories $\B_{12}, \B_{02}, \B_{01}$ together with a pseudo-functor
$${F}_{012} \colon {\B}_{12} \times {\B}_{01} \to {\B}_{02}\ ;$$
its three faces are $\B_{12}$, $\B_{02}$, and $\B_{01}$.
\item A $3$-simplex is given by:
\begin{itemize}
\item Objects ${\B}_{ij}$ for each $0 \leqslant i < j
  \leqslant 3$;
\item functors ${F}_{ijk} \colon {\B}_{jk} \times {\B}_{ij} \to {\B}_{ik}$ for
  each $0 \leqslant i < j < k \leqslant 3$;
\item a pseudo-natural transformation $\gamma_{0123}$ whose component at $a,b,c$ is
  \begin{equation*}
   \cd[@C+2em]{ F_{013}(F_{123}(a,b),c) \ar[r]^{\gamma_{0123}} & F_{023}(a, F_{012}(b,c)) }\ ;
  \end{equation*}
\end{itemize}
its four faces are $F_{123}$, $F_{023}$, $F_{013}$ and $F_{012}$.
\item A $4$-simplex is given by:
\begin{itemize}
\item Objects ${\B}_{ij}$ for each $0 \leqslant i < j \leqslant 4$;
\item functors ${F}_{ijk} \colon {\B}_{jk} \times {\B}_{ij} \to {\B}_{ik}$ for each $0 \leqslant i < j < k \leqslant 4$;
\item transformations ${\gamma}_{ijk\ell} \colon {F}_{ij\ell} \circ ({F}_{jk\ell}
  \times 1) \Rightarrow {F}_{ik\ell} \circ (1 \times {F}_{ijk})$ for each $0 \leqslant i < j < k < \ell \leqslant 4$
\item a modification $\Gamma_{01234}$ whose component at $a,b,c,d$ is
  \begin{equation*}
   \cd[@!C@C-3.5cm]{
       &&  F_{024}(F_{012}(a,b),F_{234}(c,d)) \ar[drr]^{\gamma_{0124}} \dtwocell{dd}{\Gamma_{01234}} && \\
      F_{024}(F_{023}(F_{012}(a,b),c),d) \ar[urr]^{\gamma_{0234}} \ar[d]_{F_{024}(\gamma_{0123},d)} &&&&   F_{014}(a,F_{124}
(b,F_{234}(c,d))) \\
	F_{024}(F_{013}(a,F_{123}(b,c)),d) \ar[rrrr]_{\gamma_{0134}} &&&& F_{014}(a,F_{134}(F_{123}(b,c),d)) \ar[u]_{F_{014}(a,\gamma_{1234})}
   }\ ;
  \end{equation*}
its five faces are $\gamma_{1234}$, $\gamma_{0234}$, $\gamma_{0134}$, $\gamma_{0124}$, and $\gamma_{0123}$.
\end{itemize}

\item A $5$-simplex is given by six modifications $\Gamma_{ijk\ell m}$ for $0 \leqslant i < j < k < \ell < m \leqslant 5$ as above satisfying the following equality for each $a,b,c,d,e$.
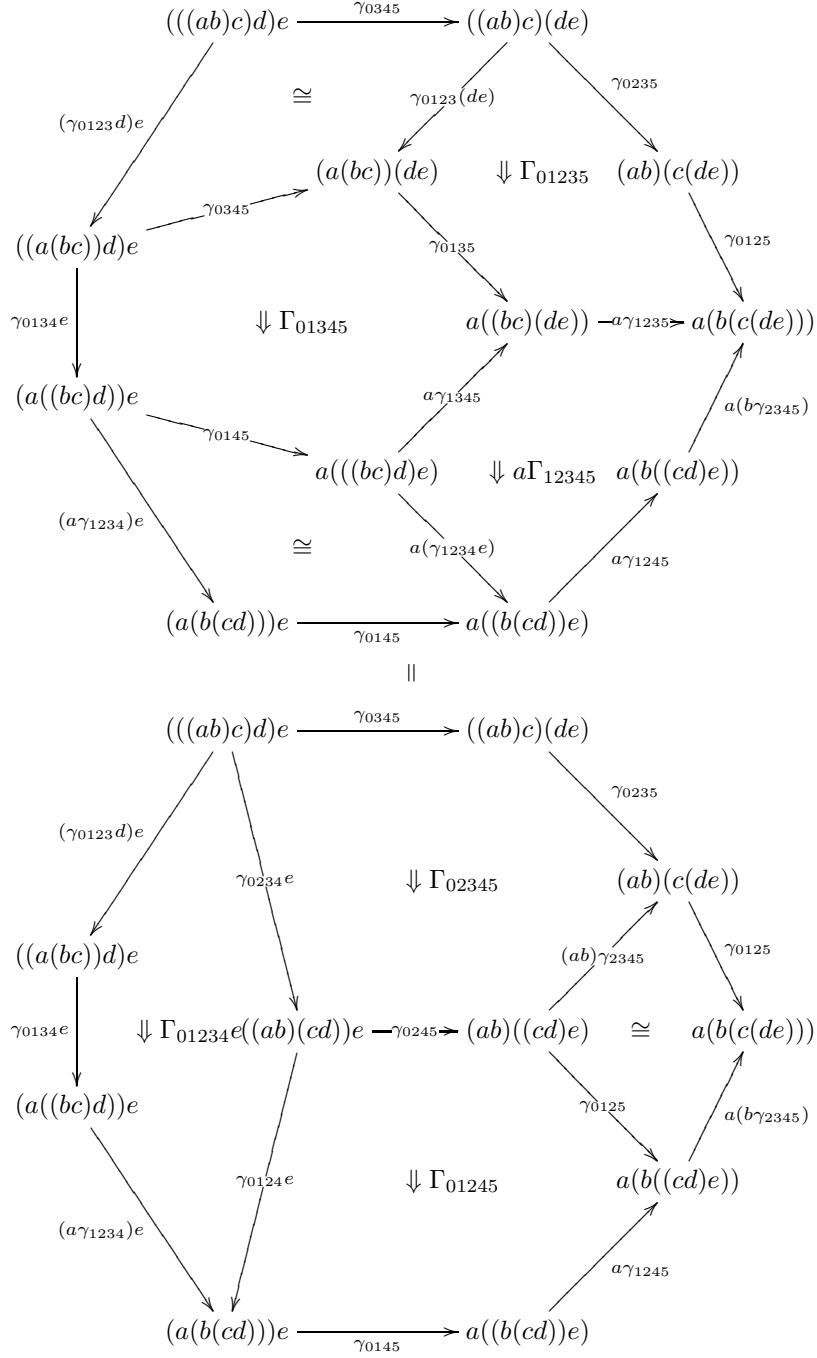
\begin{equation*}
	\def\AAA{a\Gamma_{12345}} \def\BBB{\Gamma_{02345}} \def\CCC{\Gamma_{01345}}
	\def\DDD{\Gamma_{01245}} \def\EEE{\Gamma_{01235}} \def\FFF{\Gamma_{01234}e}
	\def\abcd{\gamma_{0123}(de)} \def\abce{\gamma_{0124}e} \def\abde{\gamma_{0134}e} \def
\acde{\gamma_{0234}e} \def\bcde{a(\gamma_{1234}e)}
	\def\abcf{\gamma_{0125}} \def\abdf{\gamma_{0135}} \def\acdf{\gamma_{0235}} \def\bcdf{a\gamma_{1235}}
	\def\abef{\gamma_{0145}} \def\acef{\gamma_{0245}} \def\bcef{a\gamma_{1245}}
	\def\adef{\gamma_{0345}} \def\bdef{a\gamma_{1345}}
	\def\cdef{(ab)\gamma_{2345}}
	\def\abcdO{(\gamma_{0123}d)e} \def\adefO{\gamma_{0345}} \def\bcdeO{(a\gamma_{1234})e} \def
\abefO{\gamma_{0145}} \def\abcfO{\gamma_{0125}} \def\cdefO{a(b\gamma_{2345})}
	\def\aa{(((ab)c)d)e} \def\bb{((ab)c)(de)} \def\cc{(ab)(c(de))} \def\dd{a(b(c(de)))} \def\ee{((a(bc))d)e} \def
\ff{(a((bc)d))e} \def\gg{(a(b(cd)))e}
	\def\hh{a((b(cd))e)} \def\ii{a(b((cd)e))} \def\jj{(a(bc))(de)} \def\kk{a(((bc)d)e)} \def\ll{a((bc)(de))} \def\mm{((ab)
(cd))e} \def\nn{(ab)((cd)e)}
	\tricataxiom{20mm}
\end{equation*}
This one of the coherence axioms for a tricategory. 
It is the associahedron of dimension three, sometimes called the Stasheff polytope $K_5$ or the non-abelian 4-cocycle condition~\cite{GPS}.

The three unnamed isomorphisms come from the pseudo-naturality of 3-simplices $\gamma_{ijk\ell}$.
We have abbreviated each $F_{ijk}(F_{...}(F_{...}(ab)c)d)e$ to $(((ab)c)d)e$.
We have also chosen not to display coherence isomorphisms associated to each $F_{ijk}$.
The six faces of this simplex are $\Gamma_{12345}$, $\Gamma_{02345}$, $\Gamma_{01345}$, $\Gamma_{01245}$, $\Gamma_{01235}$,
and $\Gamma_{01234}$.
\item Higher-dimensional simplices are determined by the requirement that\\ $\N(\mathrm{Bicat})$ be $5$-coskeletal.
\end{itemize}

We still need to describe the degenerate simplices.
\begin{itemize}
\item At dimension zero, $s_0(\star) = 1$, the terminal bicategory.
\item At dimension one, $s_0(\B_{01}) \colon 1 \times \B_{01} \to \B_{01}$ and $s_1(\B_{01}) \colon \B_{01} \times 1 \to \B_{01}$ are the obvious projections.
\item Each $s_{j}({F}_{012} \colon {\B}_{12} \times {\B}_{01} \to {\B}_{02})$ for $j=0,1,2$ is a pseudo-natural transformation whose 1-cell components are identities and 2-cell components are coherence data.
\item At dimension four, $s_0(\gamma_{0123} \colon F_{013}(F_{123}(a,b),c) \to F_{023}(a, F_{012}(b,c)))$ is the unique composite of coherence 2-cells filling
\begin{equation*}
   \cd[@!C@C-2.2cm]{
       &&  F_{023}(a, F_{012}(b,c)) \ar@{=}[drr]^{\mathrm{id}}  && \\
     F_{013}(F_{123}(a,b),c) \ar[urr]^{\gamma_{0123}} \ar@{=}[d]_{\mathrm{id}} &&&&   F_{023}(a, F_{012}(b,c)) \\
	F_{013}(F_{123}(a,b),c) \ar@{=}[rrrr]_{\mathrm{id}} &&&& F_{013}(F_{123}(a,b),c) \ar[u]_{\gamma_{0123}}
   }\ ;
\end{equation*}
and the other three are similarly defined.
\item We won't display degenerate 5-simplices; they can be computed using the simplicial identities.
The equalities of pastings they describe are guaranteed to hold by coherence for bicategories.
\end{itemize}

In all of the above we have chosen to use pseudo-functors and pseudo-natural transformations rather than their lax and oplax cousins.
Those other variations might work just as well, but we haven't investigated them in any detail.

\begin{definition}
The \emph{pseudo nerve of $\mathrm{Bicat}$}, called $\N_{\mathrm{p}}\mathrm{Bicat}$, is the same as $\N\mathrm{Bicat}$ with the extra requirement that each $\gamma_{0123}$ be an equivalence and each $\Gamma_{01234}$ an isomorphism.
\end{definition}

\subsection{Skew-monoidal bicategories}\label{subsec:skewmonbicat}

In Section \ref{sec:direct_biequivalence} we showed that a simplicial map $F \colon \C \to \N\B$ was precisely a skew monoidale in a monoidal bicategory $\B$.
In that case, $F$ actually determined an extra datum $k$ and three extra axioms \eqref{stringaxiom5} \eqref{stringaxiom8} \eqref{stringaxiom9}.
Fortunately, \eqref{stringaxiom5} forced $k$ to be equal to a pasting of coherence maps already found in $\B$, and \eqref{stringaxiom8} and \eqref{stringaxiom9} were already true in any monoidal bicategory.
When we look at simplicial maps $\C$ into $\N(\mathrm{Bicat})$ or $\N_{\mathrm{p}}(\mathrm{Bicat})$ we once again find more data and axioms than we might expect.
Our approach to this data will depend on whether we want to describe monoidal bicategories, or skew-monoidal bicategories.

If our goal is to classify monoidal bicategories, we should consider simplicial maps from $\C$ into $\N_p\mathrm{Bicat}$.
In this case, because the maps in question are invertible, most of this data is over-specified and the essential extra data consists of a single equivalence $Fk \colon I \to I$ and a single isomorphism $F\delta \colon \mathrm{id}_I \Rightarrow Fk $ satisying $F\delta Fk = Fk F\delta$. Without presenting every detail: if we consider the set of all monoidal bicategories with this extra data, and also describe a suitable notion of equivalence for them, every such structure is equivalent to one where $Fk$ and $F\delta$ are trivial. Thus we have, up to equivalence, monoidal bicategories.

If our goal is to classify skew-monoidal bicategories, we should consider simplicial maps from $\C$ into $\N\mathrm{Bicat}$.
Unfortunately, we cannot use the same trick as before to eliminate this extra information because the unexpected axioms do not force the unexpected data to be trivial, even up to equivalence.
We don't yet understand what role these `extra' coherence maps might play and for the moment ask that each $F \colon \C \to \N\mathrm{Bicat}$ send the offending simplices in $\C$ to pastings of coherence data in $\mathrm{Bicat}$. 
Specifically, $k$ is mapped to the identity pseudo-natural transformation on the unit $Fi$ and $A9$ is mapped to the unique composite of coherence data with the corresponding boundary.

With this added condition in place, we define skew-monoidal bicategories by examining the data of simplicial maps $\C \to \N(\mathrm{Bicat})$.
For convenience we have used the same notation as~\cite{GPS}.

\begin{definition}
A \emph{skew-monoidal bicategory} consists of:
\begin{itemize}
\item A bicategory $\M$.
\item Two homomorphisms
$$\otimes \colon \M \times \M \to \M \quad \text{ and } \quad I \colon 1 \to \M. $$
\item Three pseudo-natural transformations
  \begin{equation*}
   \cd[@!C@C+1em]{
	\M \times \M \times \M \dtwocell{dr}{\mathfrak{a}} \ar[r]^-{\otimes \times 1} \ar[d]_{1 \times \otimes} & \M \times \M \ar[d]^{\otimes} \\
	\M \times \M \ar[r]_-{\otimes} & \M
   }
  \end{equation*}
  \begin{equation*}
   \cd[@!C@C+1em]{
	\M \ar[r]^-{I \times 1} \ar@/_1em/[dr]_-{\mathrm{id}} & \M \times \M \ar[d]^{\otimes} \\
	\dtwocell[0.6]{ur}{\mathfrak{l}} & \M \ar@{}[l]^{\phantom{\otimes}}
   }
\hspace{4em}
   \cd[@!C@C+1em]{
	\M \ar[d]_{1 \times I} \ar@/^1em/[dr]^-{\mathrm{id}} & \dtwocell[0.6]{dl}{\mathfrak{r}} \ar@{}[l]_{\phantom{I \times 1}} \\
	\M \times \M \ar[r]_-{\otimes} & \M
   }
  \end{equation*}
\item Five modifications with components:
\begin{equation*}
\cd[@!C@C-2cm]{
\dtwocell{ddrrrr}{\pi} & &(A\otimes B)\otimes (C\otimes D) \ar[drr]^{\mathfrak{a}} & & \\
((A\otimes B)\otimes C)\otimes D \ar[urr]^{\mathfrak{a}} \ar[dr]_{\mathfrak{a}\otimes D} & & & & A\otimes (B\otimes (C\otimes D) \\
& (A\otimes (B\otimes C))\otimes D \ar[rr]_{\mathfrak{a}} && A\otimes ((B\otimes C)\otimes D) \ar[ur]_{A \otimes \mathfrak{a}}&
}
\end{equation*}
\begin{equation*}
\cd[@!C@C-6pt@R+6pt]{
A\otimes B \ar[dr]_{\mathfrak{r} \otimes B} \ar[rrr]^{\mathrm{id}} & & & A\otimes B \\
\dtwocell{urrr}{\mu} & (A\otimes I)\otimes B \ar[r]_{\mathfrak{a}} & A\otimes (I\otimes B) \ar[ur]_{A \otimes \mathfrak{l}} &
}
\end{equation*}
\begin{equation*}
\cd[@!C@C-12pt]{
\dtwocell{drr}{\lambda}& I\otimes (A\otimes B) \ar[dr]^{\mathfrak{l}} & \\
(I\otimes A)\otimes B \ar[ur]^{\mathfrak{a}} \ar[rr]_{\mathfrak{l} \otimes B} &  & A\otimes B
}
\end{equation*}
\begin{equation*}
\cd[@!C@C-12pt]{
\dtwocell{drr}{\rho} & (A\otimes B)\otimes I \ar[dr]^{\mathfrak{a}} & \\
A\otimes B \ar[ur]^{\mathfrak{r}} \ar[rr]_{A \otimes \mathfrak{r}} &  & A\otimes (B\otimes I)
}
\end{equation*}
\begin{equation*}
\cd[@!C@C-6pt]{
\dtwocell{drr}{\sigma} & I\otimes I \ar[dr]^{\mathfrak{l}} & \\
I \ar[ur]^{\mathfrak{r}} \ar[rr]_{\mathrm{id}} && I
}.
\end{equation*}
\item All subject to 8 axioms.
Unnamed isomorphisms are either pseudo-naturality data or composites of coherence data in $\M$.
Empty cells are actual equalities.

\begin{equation} \label{AXIOM4_1} %AXIOM 1
	\def\AAA{A\pi} \def\BBB{\pi} \def\CCC{\pi_{BC}}
	\def\DDD{\pi} \def\EEE{\pi_{DE}} \def\FFF{\pi E}
	\def\abcd{\mathfrak{a} (DE)} \def\abce{\mathfrak{a} E} \def\abde{\mathfrak{a} E} \def\acde{\mathfrak{a} E} \def\bcde{A(\mathfrak{a} E)}
	\def\abcf{\mathfrak{a}} \def\abdf{\mathfrak{a}} \def\acdf{\mathfrak{a}} \def\bcdf{A \mathfrak{a}}
	\def\abef{\mathfrak{a}} \def\acef{\mathfrak{a}} \def\bcef{A \mathfrak{a}}
	\def\adef{\mathfrak{a}_{A(BC)}} \def\bdef{A \mathfrak{a}}
     \def\cdef{(AB)\mathfrak{a}}
	\def\abcdO{(\mathfrak{a} D)E} \def\adefO{\mathfrak{a}} \def\bcdeO{(A\mathfrak{a})E} \def\abefO{\mathfrak{a}} \def\abcfO{\mathfrak{a}} \def
\cdefO{A(B\mathfrak{a})}
	\def\aa{(((AB)C)D)E} \def\bb{((AB)C)(DE)} \def\cc{(AB)(C(DE))} \def\dd{A(B(C(DE)))} \def\ee{((A(BC))D)E} \def
\ff{(A((BC)D))E} \def\gg{(A(B(CD)))E}
	\def\hh{A((B(CD))E)} \def\ii{A(B((CD)E))} \def\jj{(A(BC))(DE)} \def\kk{A(((BC)D)E)} \def\ll{A((BC)(DE))} \def\mm{((AB)
(CD))E} \def\nn{(AB)((CD)E)}
	\tricataxiom{20mm}
\end{equation}

\begin{equation} \label{AXIOM4_2} %AXIOM 2
\begin{gathered}
	\cd[@!C@C+.3cm@R+.3cm]
	{
		A(DE) \ar[dr]|{\mathfrak{r}(DE)} \ar@/^16pt/[rrr]^{\mathrm{id}} & \dtwocell{d}{\mu} & A(I(DE)) \dtwocell[0.3]{dr}{A\lambda} \ar[r]_{A\mathfrak{l}} & A(DE) \\
		(AD)E \congcell{r} \ar[u]^{\mathfrak{a}} \ar[dr]_{(\mathfrak{r}D)E} & (AI)(DE) \dtwocell{dr}{\pi} \ar[ur]^{\mathfrak{a}} & A(I(DE)) \ar[ur]|{A(\mathfrak{l}E)} \ar[u]|{A\mathfrak{a}} \congcell{r} & (AD)E \ar[u]_{\mathfrak{a}} \\
		 & ((AI)D)E \ar[u]_{\mathfrak{a}} \ar[r]_{\mathfrak{a}E} & (A(ID))E \ar[u]_{\mathfrak{a}} \ar[ur]_{(A\mathfrak{l})E} &
	}	
    \\ \verteq \\
	\cd[@!C@C+.3cm@R+.3cm]
	{
		A(DE) \congcell{drrr} \ar@/^16pt/[rrr]^{\mathrm{id}} &&& A(DE) \\
		(AD)E \ar[rrr]|{id} \ar[u]^{\mathfrak{a}} \ar[dr]_{(\mathfrak{r}D)E} &&& (AD)E \ar[u]_{\mathfrak{a}}  \\
		& ((AI)D)E \dtwocell{ur}{\mu E} \ar[r]_{\mathfrak{a}E} & (A(ID))E \ar[ur]_{(A\mathfrak{l})E} &
	}	
\end{gathered}
\end{equation}

\begin{equation} \label{AXIOM4_3} %AXIOM 3
\begin{gathered}
	\cd[@!C@C+.3cm@R+.3cm]
	{
		&& (AB)E \ar[drr]^{\mathfrak{a}} \congcell{d} && \\
		(AB)E \ar[urr]^{\mathrm{id}} \ar[r]^{\mathfrak{a}} \ar[dr]_{(A\mathfrak{r})E} & A(BE) \congcell{d} \ar[dr]^{A(\mathfrak{r}E)} \ar[rrr]^{\mathrm{id}} &&& A(BE) \\
		& (A(BI))E \ar[r]_{\mathfrak{a}} & A((BI)E) \dtwocell{ur}{A\mu} \ar[r]_{A\mathfrak{a}} & A(B(IE)) \ar[ur]_{A(B\mathfrak{l})} &
	}	
	\\ \verteq \\
	\cd[@!C@C+.3cm@R+.3cm]
	{
		& \dtwocell[0.6]{dr}{\mu} & (AB)E \ar[drr]^{\mathfrak{a}} && \\
		(AB)E \ar[urr]^{\mathrm{id}} \ar[r]^{\mathfrak{r}E} \ar[dr]_{(A\mathfrak{r})E} & \dtwocell[0.3]{dl}{\rho E} (A(BI))E \ar[r]^{\mathfrak{a}} \ar[d]^{\mathfrak{a}E} & A((BI)E) \ar[dr]^{\mathfrak{a}} \ar[u]^{(AB)\mathfrak{l}} \congcell{rr} && A(BE) \\
		& (A(BI))E \ar[r]_{\mathfrak{a}} \dtwocell{ur}{\pi} & A((BI)E) \ar[r]_{A\mathfrak{a}} & A(B(IE)) \ar[ur]_{A(B\mathfrak{l})} &
	}	
\end{gathered}
\end{equation}

\begin{equation} \label{AXIOM4_4} %AXIOM 4
\begin{gathered}
	\cd[@!C@C+.3cm@R+.3cm]
	{
		II \ar[rrr]^{\mathrm{id}} \ar[dr]^{\mathfrak{r}I} &&& II \\
		I \congcell{r} \ar[u]^{\mathfrak{r}} \ar[dr]_{\mathfrak{r}} & (II)I \dtwocell{ur}{\mu} \dtwocell[0.3]{dr}{\rho} \ar[r]^{\mathfrak{a}} & I(II) \dtwocell[0.2]{dr}{I\sigma} \ar[ur]^{I\mathfrak{l}} & \\
		& II \ar[u]^{\mathfrak{r}} \ar[ur]_{I\mathfrak{r}} \ar@/_24pt/[uurr]_{\mathrm{id}} &&
	}	
	\\ \verteq \\
	\cd[@!C@C+.3cm@R+.3cm]
	{
		II \ar[rrr]^{\mathrm{id}} &&& II \\
		I \ar[u]^{\mathfrak{r}} \ar[dr]_{\mathfrak{r}} & & & \\
		& II \ar@/_24pt/[uurr]_{\mathrm{id}} &&
	}	
\end{gathered}
\end{equation}

\begin{equation} \label{AXIOM4_5} %AXIOM 5
\begin{gathered}	
  \cd[@!C@C+.3cm@R+.3cm]
	{
		 & ((AB)C)I \ar[rr]^{\mathfrak{a}} \dtwocell{dr}{\rho} && (AB)(CI) \ar[dr]^{\mathfrak{a}}  \congcell{dd}& \\
		 (AB)C \ar[ur]^{\mathfrak{r}} \ar[dr]_{\mathfrak{a}} \ar@/_16pt/[urrr]_{(AB)\mathfrak{r}} &&&& A(B(CI)) \\
		 & A(BC)  \ar@/_16pt/[urrr]_{A(B\mathfrak{r})} &&&
	}	
	\\ \verteq \\
	\cd[@!C@C+.3cm@R+.3cm]
	{
		 & ((AB)C)I \ar[dr]^{\mathfrak{a}I} \ar[rr]^{\mathfrak{a}}  \congcell{dd} && (AB)(CI) \ar[dr]^{\mathfrak{a}} & \\
		 (AB)C \ar[ur]^{\mathfrak{r}} \ar[dr]_{\mathfrak{a}} && (A(BC))I  \dtwocell{ur}{\pi}  \dtwocell[0.3]{d}{\rho} \ar[r]^{\mathfrak{a}} & A((BC)I) \dtwocell[0.3]{d}{A \rho} \ar[r]^{A\mathfrak{a}} & A(B(CI)) \\
		 & A(BC) \ar[ur]_{\mathfrak{r}} \ar@/_8pt/[urr]_{A\mathfrak{r}} \ar@/_16pt/[urrr]_{A(B\mathfrak{r})} &&&
	}	
\end{gathered}
\end{equation}

\begin{equation} \label{AXIOM4_6} %AXIOM 6
\begin{gathered}	
  \cd[@!C@C+.3cm@R+.3cm]
	{
		& (II)I \ar[r]^{\mathfrak{a}} \ar@/_12pt/[drr]_{\mathfrak{l}I} & I(II) \ar[dr]^{\mathfrak{l}} \dtwocell{d}{\lambda}  & \\
		II \ar[ur]^{\mathfrak{r}} \ar[dr]_{\mathfrak{l}} & & & II \\
		& I \ar[urr]_{\mathfrak{r}} \congcell{uu}  &&
	}	
  \\ \verteq \\
	\cd[@!C@C+.3cm@R+.3cm]
	{
		& (II)I \dtwocell{d}{\rho}  \ar[r]^{\mathfrak{a}}  & I(II) \ar[dr]^{\mathfrak{l}} & \\
		II \ar[ur]^{\mathfrak{r}} \ar[dr]_{\mathfrak{l}} \ar@/_12pt/[urr]_{I\mathfrak{r}} & & & II \\
		& I \ar[urr]_{\mathfrak{r}} & \congcell{uu} &
	}		
\end{gathered}
\end{equation}

\begin{equation} \label{AXIOM4_7} %AXIOM 7
\begin{gathered}	
  \cd[@!C@C+.3cm@R+.3cm]
	{
		&& (IC)(DE) \congcell{ddd} \ar[drr]^{\mathfrak{a}} \ar@/_16pt/[ddrr]_{\mathfrak{l}(DE)} && \\
		((IC)D)E \ar[urr]^{\mathfrak{a}} \ar@/_16pt/[ddrr]_{(\mathfrak{l}D)E} &&&& I(C(DE)) \dtwocell[0.3]{dll}{\lambda}  \ar[d]^{\mathfrak{l}} \\
		&&&& C(DE) \\
		&& (CD)E  \ar[urr]_{\mathfrak{a}} &&
	}			
  \\ \verteq \\
	\cd[@!C@C+.3cm@R+.3cm]
	{
		&& (IC)(DE) \ar[drr]^{\mathfrak{a}}  \dtwocell{dd}{\pi} && \\
		((IC)D)E \ar[urr]^{\mathfrak{a}} \ar[dr]^{\mathfrak{a}E} \ar@/_24pt/[ddrr]_{(\mathfrak{l}D)E} &&&& I(C(DE)) \ar[d]^{\mathfrak{l}} \\
		& (I(CD))E\ar[rr]^{\mathfrak{a}} \ar[dr]^{\mathfrak{l}E} \dtwocell[0.2]{dl}{\lambda E}  & \dtwocell{d}{\lambda}  & I((CD)E) \congcell{r} \ar[ur]^{I\mathfrak{a}} \ar[dl]_{\mathfrak{l}} & C(DE) \\
		&& (CD)E  \ar[urr]_{\mathfrak{a}} &&
	}		
\end{gathered}
\end{equation}

\begin{equation} \label{AXIOM4_8} %AXIOM 8
\begin{gathered}	
  \cd[@!C@C+.3cm@R+.3cm]
	{
		II \ar[dr]^{\mathfrak{r}I} \ar@/_24pt/[ddrr]_{\mathrm{id}} \ar[rrr]^{\mathrm{id}} &&& II \ar[d]^{\mathfrak{l}} \\
		& (II)I \ar[r]^{\mathfrak{a}} \ar[dr]^{\mathfrak{l}I} \dtwocell{ur}{\mu} \dtwocell[0.2]{dl}{\sigma I} & I(II) \ar[ur]^{I\mathfrak{l}} \dtwocell[0.3]{dl}{\lambda} \ar[d]^{\mathfrak{l}} & \congcell{l} I \\
		&& II \ar[ur]_{\mathfrak{l}} &
	}	
  \\ \verteq \\
	\cd[@!C@C+.3cm@R+.3cm]
	{
		II  \ar@/_16pt/[ddrr]_{\mathrm{id}} \ar[rrr]^{\mathrm{id}} &&& II \ar[d]^{\mathfrak{l}} \\
		& & & I \\
		&& II \ar[ur]_{\mathfrak{l}} &
	}		
\end{gathered}
\end{equation}

\end{itemize}
\end{definition}

\begin{remark}
When $\mathfrak{a}, \mathfrak{l}, \mathfrak{r}$ are equivalences and $\pi, \mu, \rho, \lambda, \sigma$ are isomorphisms this definition becomes equivalent to the usual definition of monoidal bicategory.

Beginning with a skew-monoidal bicategory with invertible coherence maps, just forget $\sigma$ and axioms \eqref{AXIOM4_4}--\eqref{AXIOM4_8} and we have exactly a monoidal bicategory.
Conversely, given a monoidal bicategory we can construct $\sigma$ according to axiom~\eqref{AXIOM4_4}, and axioms~\eqref{AXIOM4_5}--\eqref{AXIOM4_8} are implied by~\eqref{AXIOM4_1}--\eqref{AXIOM4_3} and coherence for tricategories (see Chapter 10 and Appendix C in~\cite{Gurski}).
\end{remark}

%%%% Bibliography

\end{document}